\documentclass[reqno,12pt]{amsart}
\usepackage[left=3.5cm, top=3.2cm, bottom=3.2cm, right=3.5cm]{geometry}
\usepackage{amssymb}
\usepackage{amsxtra}
\usepackage{amsmath}
\usepackage{amsfonts}
\usepackage{color}

\newcommand{\ian}{}

\numberwithin{equation}{section}

\newtheorem{thm}{Theorem}[section]
\newtheorem{Mthm}{Main Theorem}
\newtheorem{prop}[thm]{Proposition}
\newtheorem{cor}[thm]{Corollary}

\newtheorem{lem}[thm]{Lemma}
\newtheorem{de}[thm]{Definition}
\newtheorem{rem}[thm]{Remark}
\newtheorem{ex}[thm]{Example}

\def\ty{{\tilde y}}
\def\w{\omega}
\def\bx{\bold x}
\def\av{\alpha^{\vee}}
\def\tw{\widetilde W}
\def\tv{\widetilde V}
\def \bi{{i}}
\def\res{\mathop{\text{\rm res}}}

\newcommand{\eqa}{\begin{eqnarray}}
\newcommand{\eeqa}{\end{eqnarray}}
\newcommand{\beq}{\begin{equation}}
\newcommand{\eeq}{\end{equation}}
\newcommand{\nn}{\nonumber}
\newcommand{\p}{\partial}

\newcommand{\lm}{\lambda}

\newcommand{\pal}{\partial}

\newcommand{\td}{\tilde d}

\newcommand{\al}{\alpha}

\newcommand{\ta}{\theta}

\def \om{\mu}

 \def \dsum{\displaystyle\sum}

\newcommand{\pf}{\noindent{\it Proof. \ }}
\newcommand{\epf}{$\quad$\hfill
\raisebox{0.11truecm}{\fbox{}}\par\vskip0.4truecm}

\newenvironment{prfMM}{\noindent {\it Proof of the Main Theorem\,1.} \ }{\hfill $\Box$}

\begin{document}
\title[]{Extended affine Weyl groups of BCD-type: their Frobenius manifolds and Landau--Ginzburg superpotentials}
\author[]{
{Boris Dubrovin \ \ Ian A.B. Strachan \ \ Youjin Zhang \  \ Dafeng Zuo}}

\address{Dubrovin\footnote{Deceased on March 19, 2019.} , SISSA, Via Bonomea, 265, I-34136 Trieste, Italy and
Steklov Math. Institute, Moscow, Russia}

\address{Strachan, School of Mathematics and Statistics, University of Glasgow, UK}

\address{Zhang, Department of Mathematical Sciences, Tsinghua
University, Beijing 100084, P. R. China}
\address{Zuo, School of Mathematical Science,
University of Science and Technology of China,
 Hefei 230026,P.R.China}

\email{\mbox{ian.strachan@glasgow.ac.uk,\quad youjin@mail.tsinghua.edu.cn, }
\newline\mbox {\hskip3cm dfzuo@ustc.edu.cn}}
\date{}

\subjclass[2010]{Primary 53D45}

\keywords{extended affine Weyl group, Frobenius manifold, LG superpotential}

\date{\today}

\begin{abstract}
For the root systems of type $B_l, C_l$ and $D_l$, we generalize the
result of \cite{DZ1998} by showing the existence of Frobenius
manifold structures on the orbit spaces of the extended affine Weyl
groups that correspond to any vertex of the Dynkin diagram instead
of a particular choice  made in \cite{DZ1998}. It also depends on certain
additional data. We also construct Landau--Ginzburg superpotentials for these
Frobenius manifold structures.
\end{abstract}
\maketitle

\tableofcontents

\section{Introduction}

\subsection{Background}

\ian{In \cite{DZ1998}, the precursor to this paper, the Frobenius manifold structures on the orbit spaces $\mathbb{C}^{l+1}/{\widetilde{W}}$, where
$\widetilde{W}$ are certain extended affine Weyl groups, were constructed. In particular, the construction depended on a choice of
a specific node on the Dynkin diagram of the underlying Weyl group $W\,.$ In the case $W\cong A_l$ this marked node could be taken
to be an arbitrary node, but for the remaining cases a specific node was used. In addition, again for the case $W\cong A_l$\,,
a Landau--Ginzburg (or LG) superpotential construction of the Frobenius manifold structure was given.
\vskip 2mm
The aim of this paper is to complete the construction for the underlying Weyl groups $W\cong B_l\,,C_l\,,D_l\,,$ providing:
\begin{itemize}
\item[$\bullet$] a construction of the Frobenius manifold structure on the orbit space for an {\sl arbitrary} marked node;
\item[$\bullet$] a LG superpotential construction for all these manifolds.
\end{itemize}
A remaining problem is to extend the construction to the few remaining exceptional Weyl groups, and we hope to return to this problem in a
 later paper.}

\ian{The construction rests on the proof of a Chevalley-type theorem for invariant polynomials for the extended-affine Weyl groups,
together with a Saito-construction for the flat coordinates on the orbit space. Besides this, there is another construction,
which is provided by a LG superpotential construction, and involves a refinement of the genus zero  Hurwitz space theory to the
case of cosine-Laurent polynomials. These two separate constructions
 can then be shown to result in isomorphic Frobenius manifolds.
}

\subsection{The main results}

Let $R$ be an irreducible reduced root system defined in an
$l$-dimensional Euclidean space $V$ with the Euclidean inner product $(~,~)$.
We fix a basis of simple roots $\al_1,\dots,\al_l$  and
denote by  $\alpha_j^{\vee},\ j=1,2,\cdots,l$ the corresponding
coroots. The Weyl group $W$ is generated by the reflections
\beq
{\bx}\mapsto \bx-(\al_j^\vee,\bx)\al_j,\quad \forall\, {\bx}\in V,\ j=1,\dots,l.
\eeq
Recall that the \emph{Cartan matrix} of the root system has integer entries $A_{ij}=\left( \alpha_i, \alpha_j^\vee\right)$
satisfying $A_{ii}=2$, $A_{ij}\leq 0$ for $i\neq j$. The
semi-direct product of $W$ by the lattice of coroots yields the
affine Weyl group $W_a$ that acts on $V$ by the affine
transformations \beq \bx\mapsto w(\bx)+\sum_{j=1}^l m_j\av_j,\quad
w\in W,\ m_j\in \mathbb{Z}. \eeq
We denote by $\omega_1,\dots,\omega_l$ the fundamental weights  defined by the relations
\beq
(\omega_i,\al_j^\vee)=\delta_{ij},\quad i, j=1,\dots,l.
\eeq

Note that the root system $R$ is one of the type $A_l, B_l, C_l, D_l, E_6, E_7,
E_8$, $F_4, G_2$. In what follows the Euclidean space $V$
and the basis $\al_1,\dots, \al_l$ of the simple roots will be defined as in Plate I-IX of \cite{bb}.
Let us fix a simple root $\al_k$ and define an
extension of the affine Weyl group $W_a$ in a similar way as was done in \cite{DZ1998}.
\begin{de} The extended affine Weyl group $\tw=\tw^{(k)}(R)$ acts on the extended space
$$
\tv=V\oplus \mathbb{R}\,,
$$
and is generated by the transformations
\begin{equation}
x=(\bx,x_{l+1})\mapsto (w(\bx)+\sum_{j=1}^l m_j\av_j, \
x_{l+1}),\quad w\in W,\ m_j\in \mathbb Z,
\end{equation}
and
\begin{equation}
x=(\bx,x_{l+1})\mapsto (\bx+\gamma\, \w_k,\ x_{l+1}-\gamma).
\end{equation}
Here $1\le k\le l$, $\gamma=1$ except for the cases when $R=B_l, k=l$ and $R=F_4, k=3$ or $k=4$, in these three
cases $\gamma=2$.
\end{de}

The above definition of the extended affine Weyl group coincides with the one given
in \cite{DZ1998} for the particular choice of $\al_k$ made there.
We note that in the cases for which $\gamma=1$ the numbers $\frac12 (\al_k,\al_k)$ are integers, while
for the three exceptional cases $\frac12 (\al_k,\al_k)=\frac12$.
\ian{Since we will only be considering the non-exceptional cases, we will take $\gamma=1$ for the remainder of the paper.}

Coordinates $x_1,\dots,x_l$ may be introduced on the space $V$ via the expression
\beq
\bx=x_1 \al_1^\vee+\dots+x_l \al_l^\vee.
\eeq
Let $f=\det(A_{ij})\,,$ the determinant of the Cartan matrix of the root system $R$.
\begin{de}[\cite{DZ1998}] ${\mathcal A}={\mathcal A}^{(k)}(R)$ is
the ring of all $\ \tw$-invariant Fourier polynomials of the form
$$
\sum_{m_1,\dots,m_{l+1}\in\, \mathbb{Z}} a_{m_1,\dots,m_{l+1}} e^{2\pi i
(m_1 x_1+\cdots+m_l x_l+\frac1{f} m_{l+1} x_{l+1})}
$$
bounded in the limit
\begin{equation}\label{bei-2}
\bx=\bx^{0}-i \w_k\tau,\quad x_{l+1}=x_{l+1}^{0}+i \tau,\quad
\tau\to +\infty
\end{equation}
for any $\ x^{0}=(\bx^{0},x_{l+1}^{0})$.
\end{de}

\ian{Condition (\ref{bei-2}) is essential for this construction.\footnote{\ian{From the invariance with respect to $\widetilde{W}$,
it easily follows (using theorem [B] in \cite{DZ1998, bb}) that any $f(x)$ can be represented
as a polynomial in $\ty_1(x),\dots,\ty_{l+1}(x), {\ty_{l+1}(x)}^{-1}$. The boundedness condition \eqref{bei-2} is equivalent to
the restriction to invariants which extend over the locus $\ty_{l+1}=0$. With this condition, $\mathcal{M}=\mbox{Spec}(\mathcal{A})$ is a partial
compactification of the full complex
orbit space.}}
Constraints also appear in the more abstract constructions in \cite{LO2,slodowy,wir} and an open problem is
 to relate these seemingly different sets of constraints.}

We introduce a set of numbers
\beq\label{zh-16}
d_j=(\omega_j, \omega_k),\quad j=1,\dots,l
\eeq
and define the following Fourier polynomials \cite{DZ1998}
\eqa
&&\ty_j(x)=e^{2\pi i d_j x_{l+1}} y_j(\bx),\quad j=1,\dots,l,\label{ip-a}\\
&&\ty_{l+1}(x)=e^{\frac{2\pi i}{\gamma} x_{l+1}}.\label{ip-b}
\eeqa
Here
$y_1(\bx),\dots,y_l(\bx)$ are the basic $W_a$-invariant Fourier
polynomials defined by
\beq y_j(\bx)=\frac1{n_j} \sum_{w\in W}
e^{2\pi i (\omega_j, w(\bx))},\quad n_j=\#\{w\in W|e^{2\pi i
(\omega_j,w(\bx))}=e^{2\pi i (\omega_j,\bx)}\}.\label{fip}
\eeq
It was shown in \cite{DZ1998} that for {certain} particular choices of the
simple root $\al_k$, a Chevalley-type theorem holds true for the
ring ${\mathcal A}$, i.e., it is isomorphic to the polynomial ring generated by
$\ty_1,\dots,\ty_{l+1}$, and thus the orbit space of the extended affine Weyl
group $\widetilde{W}$ \emph{defined} as
$\mathcal{M}={\rm Spec} \,{\mathcal A}$  is an affine algebraic variety of dimension
$l+1$. In \cite{DZ1998} it was further proved that on such an orbit space there
exists a Frobenius manifold structure whose potential is a
polynomial of $t^1,\dots, t^{l+1}, e^{t^{l+1}}$. Here $t^1,\dots,
t^{l+1}$ are the flat coordinates of the Frobenius manifold. For
the root system of type $A_l$, there is in fact no restrictions on
the choice of $\al_k$. However, for the root systems of
type $B_l, C_l, D_l, E_6, E_7, E_8, F_4, G_2$ there is
only one choice for each system.

In \cite{slodowy} Slodowy pointed out that the Chevalley-type theorem of
\cite{DZ1998} is a consequence of the results of Looijenga and Wirthm\"uller \cite{LO1,LO2, wir},
and in fact it holds true for any choice of the base element $\al_k$, or
equivalently, for any fixed vertex of the Dynkin diagram. Hence:

\begin{thm}[\cite{slodowy,wir,LO1,LO2}]\label{thm1}
The ring $\mathcal{A}$ is isomorphic to the ring of polynomials of
$\tilde{y}_1(x)$, $\cdots$, $\tilde{y}_{l+1}(x)$.
\end{thm}
A natural question, as was raised in \cite{DZ1998,
slodowy}, is {\it whether the geometric structures revealed in
\cite{DZ1998} also exist on the orbit spaces of the extended affine
Weyl groups for an arbitrary choice of the root $\al_k$?} The purpose of the
present paper is to give an affirmative answer to this question for
the root systems of type $B_l, C_l$ and $D_l$ {\ian{(recall that for the root system of type $A_l$, the question was
 already answered affirmatively in \cite{DZ1998}). We will therefore concentrate on these cases,
 the proofs of which turn out to work in a very similar manner.}}
In Sec.\ref{sec-2} we give an
elementary proof of Theorem \ref{thm1} {\ian{for
the root systems of type $B_l, C_l$ and $D_l$} that is based on the proof of
the Chevalley type theorem given in \cite{DZ1998}.

Let ${\mathcal M}$ be the
orbit space of the extended affine Weyl group $\tw^{(k)}(C_l)$ and $\widetilde{\mathcal M}$
{the universal} covering of ${\mathcal M}\setminus\{\ty_{l+1}=0\}$. In Sec.\ref{sec-3},
firstly we introduce
an indefinite metric \footnote{{As is common in the Frobenius manifold literature, we use the word metric to denote a complex-valued, symmetric, non-degenerate, bilinear form.}}
 $(~,~)^{\sptilde}$ on
\ian{${T^*}\widetilde{V}_\mathbb{C}\,,$ the complexification $\widetilde{V}_\mathbb{C}=\widetilde{V}\otimes_\mathbb{R}\mathbb{C}$ of the extended space $\widetilde{V}=V\oplus\mathbb{R}\,.$ This is defined by}
\eqa (d x_s,dx_n)^{\sptilde}=\frac{s}{4\pi^2},\quad  (d x_s,
dx_{l+1})^{\sptilde}=0,\quad (dx_{l+1},dx_{l+1})^{\sptilde}=-\frac1{4 k\pi^2} \eeqa
for $1\le s\le n\le l.$ The projection
{\beq
\Pr: \widetilde V\to \widetilde{\mathcal M},\qquad (x_1,\cdots,x_{l+1})\mapsto (y^1,\cdots,y^{l+1})\nn
\eeq}
induces a symmetric {bilinear form on $T^{*}\widetilde{\mathcal{M}}$} {defined by the matrix}
\begin{equation}\label{metric2}
g^{ij}(y):=\dsum_{a,b=1}^{l+1}\dfrac{\p y^i}{\p x_a}\dfrac{\p
y^j}{\p x_b}(dx_a,dx_b)^{\sptilde},
\end{equation}
{in the coordinates}
$ y^1=\ty_1,\dots, y^l=\ty_l,\ y^{l+1}=\log \ty_{l+1}=2 \pi i x_{l+1}.$
{We will show in Sec.\,\ref{sec-3} that the functions $g^{ij}(y)$ are homogenous polynomials of $y^1, \dots, y^l, e^{y^{l+1}}$, so they are well defined on $\mathcal{M}$. Denote
\beq\label{zh-2}
\Sigma=\{(y^1, \dots, y^l, e^{y^{l+1}})\in \mathcal{M}\,|\,\det(g^{ij}(y))=0\}.
\eeq
Then $\Sigma$ called the \emph{discriminant} of the extended affine Weyl  group $\tw^{(k)}(C_l)$ is an algebraic subvariety of $\mathcal{M}$. It consists of the orbits of points $({\bf x}, x_{l+1})\in \widetilde V$ on the \emph{mirrors} of the group
with $(\beta, {\bf x})\in \mathbb Z$ for some positive root $\beta$, the details of which will be given in Section \ref{sec-3} below.}
Afterwards, we will construct another symmetric bilinear form on $T^{*}\mathcal{M}$ by
\beq\label{1.14} \eta^{ij}(y):= {\mathcal L}_e g^{ij}(y).
\eeq
Here the vector field $e$ has the form
\beq
e=\sum_{j=k}^{l} c_j \frac{\pal}{\pal y^j},
\eeq
it depends on the choice of an integer $m$ in the range $0\le m\le l-k$.
Namely, for a given $m$ the coefficients $c_k,\dots, c_l$ are defined by the generating function
\[ \sum_{j=k}^l c_j u^{l-j} =(u+2)^m (u-2)^{l-k-m}.\]
{The symmetric bilinear forms $(\eta^{ij})$ is non-degenerate on $\mathcal{M}\setminus {\Sigma_1\cup \Sigma_2}$,
{where
\[\Sigma_1=\{f_1=0\}\subset{\mathcal M},\quad \Sigma_2=\{f_2=0\}\subset{\mathcal M}\]
are the loci \footnote{{Vanishing of $f_1$ or $f_2$ is equivalent to the condition $2(x_j-x_{j-1})\in \mathbb Z$ for some $j$. Thus the loci $\Sigma_1, \Sigma_2$ consist of the orbits of points belonging to the mirrors of the affine Weyl group corresponding to the long roots. Hence $\Sigma_1\cup \Sigma_2 \subset \Sigma$. We are grateful to the anonymous referee for this observation. }} of zeros of the following $\tw^{(k)}(C_l)$-invariant polynomials
\beq\label{zh-12}
f_1=e^{2\pi i k\, x_{l+1}}\prod_{j=1}^l\cos^2\pi(x_j-x_{j-1}),\quad
f_2=e^{2\pi i k\, x_{l+1}}\prod_{j=1}^l\sin^2\pi(x_j-x_{j-1})
\eeq
with $x_0=0$,
}
and it gives the flat metric of the Frobenius manifold structure that we are to construct.
Denote
\[\mathcal{M}_{k,m}(C_l):=\left\{\begin{array}{l}
\mathcal{M} \setminus\{\tilde{y}_{l+1}=0\}\cup\Sigma_1\cup\Sigma_2, \quad \textrm{when}\ 0< m<l-k;\\
\mathcal{M} \setminus\{\tilde{y}_{l+1}=0\}\cup \Sigma_1, \quad \textrm{when}\ m=0;\\
\mathcal{M} \setminus\{\tilde{y}_{l+1}=0\}\cup \Sigma_2, \quad \textrm{when}\ m=l-k.
\end{array}\right.\]
Then we will show the following result.}

\begin{Mthm}\label{mt1}
For any fixed integer $0\le m\le l-k$, there exists a unique Frobenius
manifold structure of charge $d=1$ on the orbit space
$\mathcal{M}_{k,m}(C_l)$ of $\tw^{(k)}(C_l)$ such that
\begin{enumerate}
\item the invariant flat metric and the intersection form of the Frobenius manifold structure coincide
with the
metrics $(\eta^{ij}(y))$ (see eq. \eqref{1.14}) and $(g^{ij}(y))$ (see eq. \eqref{metric2}) respectively;
\item the unity and the Euler vector fields have the form
\beq\label{zh-18-1}
e=\dsum_{j=k}^{l} c_j \frac{\pal}{\pal y^j}
\eeq
and
\beq\label{zz1-1}
E=\sum_{\alpha=1}^l \frac{d_\alpha}{d_k} y^\alpha\dfrac{\p}{\p y^\alpha}+\frac{1}{d_k} \dfrac{\p}{\p y^{l+1}},
\eeq
where $d_1,\dots,d_{l}$ are defined in \eqref{add1.4};
\item in the flat coordinates $t^1$, \dots, $t^{l+1}$ of the metric \eqref{1.14} defined on certain covering of $\mathcal{M}_{k,m}(C_l)$ the Frobenius manifold structure is
polynomial in
$t^1,\cdots, t^{l+1},\dfrac{1}{t^{l-m}}, \dfrac{1}{t^{l}}, e^{t^{l+1}}$.
In these coordinates
\beq\label{zh-18}
e=\dfrac{\p}{\p t^k}
\eeq
and
\begin{equation}
E=\dsum_{\alpha=1}^{l}\td_\alpha t^\alpha \dfrac{\p}{\p t^\alpha}
+\frac1{k}\dfrac{\p}{\p t^{l+1}},\label{zz1}
\end{equation}
where $\td_1,\dots,\td_{l}$ are defined in (\ref{zh18a})--(\ref{zh18c}).
\end{enumerate}
\end{Mthm}
 {We prove the above theorem in Sec.\,\ref{sec-6}. Let us note that the monodromy group of the Frobenius manifold $\mathcal{M}_{k,m}(C_l)$ (for the definition see \cite{Du1}) is isomorphic to $\tw^{(k)}(C_l)$.}

In Sec.\ref{sec-6} we further show that for the root systems of type $B_l$ and $D_l$
we can apply a similar construction as the one for
the root system of type $C_l$. The resulting Frobenius manifolds are
isomorphic to those obtained from the root system of type $C_l$.

 {For the case of $A_l$ an alternative construction of the Frobenius manifold structure was given in \cite{DZ1998}.
This structure was given in terms of a LG superpotential construction. In particular, it was shown}
that the extended affine Weyl group $\widetilde{W}^{(k)}(A_l)$
describes the monodromy of roots of trigonometric polynomials {\ian - the superpotential -} with a given
bidegree being of the form
\beq \lambda(\varphi)=e^{\bi k\varphi}+a_1e^{\bi (k-1)\varphi}+\cdots+a_{l+1}e^{\bi (k-l-1)\varphi},\quad a_{l+1}\ne 0.\nn\eeq
A natural question is
{\it does there exist a similar construction for the root systems of
type $B_l, C_l$ and $D_l$?} In Sec.\ref{sec-7},
let us denote by $\mathfrak{M}_{k,m,n}$ the space of a particular class of
cosine Laurent series  {\ian{or superpotentials}} of one variable with a given tri-degree $(2k,2m,2n)$ being of the form
\eqa
\lambda(\varphi)=\left(\cos^{2}(\varphi)-1\right)^{-m}
\dsum_{j=0}^{k+m+n} a_j \cos^{2(k+m-j)}(\varphi),\nn
\eeqa
where all $a_j \in \mathbb{C}$, $m,n\in \mathbb{Z}_{\geq 0}$, $k\in \mathbb{N}$, and \ian{ the coefficients $a_0, \dots, a_{k+m+n}$ satisfy the conditions given in
\eqref{cnd-a-1}--\eqref{cnd-a-3}.
The space $\mathfrak{M}_{k,m,n}$ carries a natural structure of Frobenius manifold. Its invariant inner product
$\eta$ and the intersection form $g$ of two vectors $\p'$, $\p''$
tangent to $\mathfrak{M}_{k,m,n}$ at a point $\lambda(\varphi)$ can be defined by
the formulae \eqref{fm2.3} and \eqref{fm2.4}.
We will show that (see Theorem\,\ref{Main2})
\begin{Mthm}\label{mt2}
 The Frobenius manifolds $\mathcal{M}_{k,m}(C_{k+m+n})$ and  $\mathfrak{M}_{k,m,n}$ are locally isomorphic.
\end{Mthm}}
{A function involved in the representation of the form \eqref{fm2.3}, \eqref{fm2.4} of }
the flat pencil of metrics on the Frobenius manifold is called a \emph{LG} \emph{superpotential} of the Frobenius manifold. Observe that the multiplication law on the tangent spaces to the Frobenius manifold can also be expressed in terms of the LG superpotential (see eq. \eqref{fm2.6} below).

 Some concluding remarks are given in the
last section.


\section{The proof of Theorem \ref{thm1} for the root systems of type $B_l, C_l, D_l$}\label{sec-2}

In this section, we give an elementary proof of the Theorem \ref{thm1} for the root systems of type
 $B_l, C_l$ and $D_l$ for any fixed vertex of the Dynkin diagram. To this end,
we first write down the explicit expressions of the invariant
Fourier polynomials ${\tilde y}_j(x)$ that are defined by
(\ref{ip-a}), (\ref{ip-b}) for these root systems with the fixed simple root $\al_k$.
We then prove the theorem by using an approach that is similar to the one used in
\cite{DZ1998}.

For the root system of type $B_l$, the numbers $d_j$ defined in (\ref{zh-16}) have the
values
\beq
d_i=i,\ 1\le i\le k,\quad d_j=k,\ k+1\le j\le l-1,\quad d_l=\frac{k}2,\label{addDZ2.1}\\
\eeq
for $k<l$ and
\beq
d_i=\frac{i}2,\ 1\le i\le l-1,\quad d_k=\frac{l}4
\label{addDZ2.2}\eeq
for $k=l$.  The $W_a$-invariant Fourier polynomials
$y_1(\bx),\dots, y_l(\bx)$ defined in (\ref{fip}) have the
expressions \cite{hoffman}
\eqa
&&{y}_j({\bx})=\sigma_j(\xi_1,\cdots,\xi_l),\quad j=1,\dots,l-1,\label{add1.10-a}\\
&& y_l(\bx)=
\prod_{j=1}^{l} \left({e^{i\pi v_j }}+e^{-i\pi  v_j }\right)\label{add1.10-b},
\eeqa
where
\eqa
&&v_1=x_1,\quad v_m=x_m-x_{m-1},\quad 2\le m\le l-1,\nn\\
&&v_{l}=2 x_l-x_{l-1},\label{add1.10-c}\\
&&\xi_j=e^{2\,i\pi v_j}+e^{-2\,i\pi v_j},\quad 1\le j\le l.\nn
\eeqa
Here and henceforth the functions $\sigma_j(\xi_1,\dots,\xi_l)$ denote the $j$-th
elementary symmetric polynomial of $\xi_1,\cdots,\xi_l$ defined by
\beq
\prod_{j=1}^l (z+\xi_j)=\sum_{j=0}^l \sigma_j (\xi_1,\dots,\xi_l) z^{l-j}.
\eeq

For the root system of type $C_l$, the numbers $d_j$ are given by
\beq\label{add1.4}
d_1=1,\dots, d_{k-1}=k-1,\ d_j=k,\quad k\le
j\le l.
\eeq
The $W_a$-invariant Fourier polynomials
$y_1(\bx),\dots, y_l(\bx)$ defined in (\ref{fip}) have the
expressions
\beq\label{add1.5}
y_j({\bx})=\sigma_j(\xi_1,\cdots,\xi_l).
\eeq
Here $\xi_j$ are defined by
\beq
\xi_j={e^{2 i\pi \left( x_{{j}}-x_{{j-1}} \right)}}+
{e^{-2 i \pi \left( x_{{j}}-x_{{j-1}} \right) }},\quad x_0=0,\ 1\le j\le l.\nn
\eeq

For the root system of type $D_l$, we have\\
\noindent i) \eqa
&&d_j=j,\ 1\le j\le k,\quad d_j=k,\ k+1\le j\le l-2,\label{addDZ2.9}\\
&&d_j=\frac{k}2, \ j=l-1,l \label{addDZ2.10}
\eeqa
for \ $k\le l-2$; and\\
\noindent ii)
\beq
d_j=\frac{j}2,\ 1\le j\le l-2, \quad d_{l-1}=\frac{l}4, \quad d_l=\frac{l-2}4
\eeq
for \ $k= l-1$; and\\
\noindent iii)
\beq
d_j=\frac{j}2,\ 1\le j\le l-2, \quad d_{l-1}=\frac{l-2}4, \quad d_l=\frac{l}4
\eeq
for \ $k= l$.\,
The basis of the $W_a$-invariant Fourier
polynomials defined in (\ref{fip}) has the form
\eqa
&&{y}_j({\bx})=\sigma_j(\xi_1,\cdots,\xi_l),\quad j=1,\dots,l-2,\nn\\
&& y_{l-1}({\bf x})=\frac12\left(\prod_{j=1}^l \left(e^{i\pi v_j}+ e^{-i\pi v_j}\right)+
\prod_{j=1}^l \left(e^{i\pi v_j}- e^{-i\pi v_j}\right)\right), \label{add1.15-a} \\
&& y_{l}({\bf x})=\frac12\left(\prod_{j=1}^l \left(e^{i\pi v_j}+ e^{-i\pi v_j}\right)-
\prod_{j=1}^l \left(e^{i\pi v_j}- e^{-i\pi v_j}\right)\right),\nn
\eeqa
where
\eqa
&&v_1=x_1,\quad v_m=x_m-x_{m-1},\quad 2\le m\le l-2,\nn\\
&&v_{l-1}=x_l+x_{l-1}-x_{l-2},\quad v_l=x_{l-1}-x_{l},\label{add1.15-b}\\
&&\xi_j=e^{2 i\pi v_j}+e^{-2 i\pi v_j},\quad 1\le j\le l.\nn
\eeqa

\vskip 0.3truecm
\noindent {\em Proof of Theorem \ref{thm1} for the root system $R=B_l, C_l, D_l$.}\quad
From the explicit expressions of the Fourier polynomials ${\tilde y}_1(x),
\dots,{\tilde y}_{l+1}(x)$, it is not difficult to see that they are
$\tw^{(k)}(R)$-invariant. So in order to prove the theorem, we only need to
show that any element $f(x)$ of the ring $\mathcal{A}$ can be expressed as a
polynomial of ${\tilde y}_1(x),\dots,{\tilde y}_{l+1}(x)$.
By using the fact that the ring of $W_a$-invariant Fourier polynomials is isomorphic
to the polynomial ring generated by
$y_1(\bx),\dots,y_l(\bx)$ and by using the ${\tw}$-invariance
of the function $f(x)\in \mathcal{A}$, we can
represent it as a polynomial of  $\
\ty_1(x),\dots,\ty_l(x),\ty_{l+1}(x),\ty_{l+1}^{-1}$.
Assume
$$
f(x)=\sum_{n\ge -N} \ty_{l+1}^n P_n(\ty_1(x),\dots,\ty_{l}(x)),
$$
and that the polynomial
$\ P_{-N}(\ty_1(x),\dots,\ty_{l}(x))$\
does not vanish identically for a certain positive integer $N$.
From the definition of the functions ${\tilde y}_j(x)$ we know that
in the limit (\ref{bei-2}) we have
\begin{equation}
y_j(\bx)=e^{2\pi d_j\tau}[y_j^{0}(\bx^{0})+\mathcal O(e^{-2\al\pi \tau})],
\quad j=1,\dots,l,
\end{equation}
where $\al$ is a certain positive integer and the
expressions of the functions $\ y_j^{0}(\bx^0)$ will be given below.
So in the limit (\ref{bei-2}) the function $f(x)$ behaves as
$$
f(x)=e^{\frac{2\pi}{\gamma} N\tau-\frac{2\pi i}{\gamma} N x_{l+1}^0}
[P_{-N}(\ty_1^{0}(x^0),\dots,\ty_{l}^{0}(x^0))+ \mathcal O(e^{-2\beta\pi \tau})]
$$
for a certain positive integer $\beta$  and
$$
\ty_j^{0}(x^0)=e^{2\pi i d_j x_{l+1}^{0}}\ y_j^{0}(\bx^0),\quad j=1,
\cdots,l.
$$
Since the function $f(x)$ is bounded for $\ \tau\to +\infty$, we must have
$$
P_{-N}(\ty_1^{0}(x^{0}),\dots,\ty_{l}^{0}(x^{0}))\equiv 0
$$
for any $x^{0}=(\bx^0,x_{l+1}^0)$.
This leads to a contradiction to the algebraic independence of the functions
${\tilde y}_1^{0},\dots,{\tilde y}_l^{0}$, a fact that we will now prove, case-by-case, for the root systems of type $B_l, C_l$ and $D_l$.

\noindent{i)} For the root system of type $B_l$ with $1\le k\le l-1$,
\eqa\label{add1.23}
&&y_j^0(\bx^0)=\rho_j,\quad
 j=1,\cdots,k,\nn\\
&& y_s^0(\bx^0)=\rho_k\rho_s,\quad s=k+1,\cdots,l-1,\nn\\
&&{y_l^0(\bx^0)=\sqrt{\rho_k}\rho_l\,,}\nn
\eeqa
where the functions $\rho_i$ are defined by
\eqa\label{add1.24}
&&\rho_j=\sigma_j(e^{2\pi i v_1^0},\cdots,e^{2\pi i v_k^0}),\quad j=1,\cdots,k,\nn\\
&&\rho_s=\sigma_{s-k}(\xi_{k+1}^0,\cdots,\xi_l^0), \quad s=k+1,\cdots,l-1,\nn\\
&&{\rho_l=\prod_{s=k+1}^l\left(e^{i\pi v_s^0}+e^{-i\pi v_s^0}\right)}\nn
\eeqa
with
\eqa
&&\xi^0_m=e^{2\pi iv_m^0}+e^{-2\pi iv_m^0},\quad m=1,\cdots,l,\nn\\
&&v_1^0=x^0_1,\quad  v_j^0=x^0_j-x_{j-1}^0,\ 2\le j\le l-1,\quad
 v_l^0=2 x_{l}^0-x^0_{l-1}.\nn
\eeqa
With these we obtain
\beq\label{add1.25}
\det\left(\dfrac{\p y_i^0(\bx^0)}{\p \rho_j}\right)=\rho_k^{l-k-1}\sqrt{\rho_k}.
\eeq

When $k=l$, we have
\eqa
&&y_j^0(\bx^0)=\rho_j=\sigma_j(e^{2\pi i v_1^0},\cdots,e^{2\pi i v_l^0}),\quad
 j=1,\cdots,l-1,\nn\\
&& y_l^0(\bx^0)=\rho_l =\prod_{s=1}^l e^{\pi i v_s^0},\nn\\
&&\det(\dfrac{\p y_i^0(\bx^0)}{\p \rho_j})=1.
\eeqa

\noindent{ii)} For the root system of type $C_l$,
\eqa\label{add1.20}
&&y_j^0(\bx^0)=\rho_j,\quad j=1,\cdots,k,\nn\\
&&y_s^0(\bx^0)=\rho_k\rho_s,\quad s=k+1,\cdots,l,\nn
\eeqa
where the functions $\rho_j$ are defined by
\eqa\label{add1.21}
&&\rho_j=\sigma_j(e^{2\pi i v_1^0},\cdots,e^{2\pi i v_k^0}),\quad j=1,\cdots,k,\nn\\
&&\rho_s=\sigma_{s-k}(\xi_{k+1}^0,\cdots,\xi_l^0), \quad s=k+1,\cdots,l\nn
\eeqa
with
\eqa
&&\xi^0_m=e^{2\pi iv_m^0}+e^{-2\pi iv_m^0},\nn\\
&&v_1^0=x^0_1,\ v_m^0=x^0_m-x_{m-1}^0,\quad m=2,\cdots,l.\nn
\eeqa
With these we obtain
\beq\label{add1.22}
\det\left(\dfrac{\p y_i^0(\bx^0)}{\p \rho_j}\right)=(\rho_k)^{l-k}.
\eeq

\noindent{iii)} For the root system of type $D_l$ with $k\le l-2$,
\eqa\label{add1.26}
&&y_j^0(\bx^0)=\rho_j,\quad
 j=1,\cdots,k,\nn\\
&&y_s^0(\bx^0)=\rho_k\rho_s,\quad s=k+1,\cdots,l-2,\nn\\
&&y_{l-1}^0(\bx^0)=\frac12\sqrt{\rho_k}\left(\rho_l+\rho_{l-1}\right),
\quad y_l^0(\bx^0)=\frac12\sqrt{\rho_k}\left(\rho_l-\rho_{l-1}\right)\nn
\eeqa
where the functions $\rho_j$ are given by
\eqa\label{add1.27}
&&\rho_j=\sigma_j(e^{2\pi i v_1^0},\cdots,e^{2\pi i v_k^0}),\quad j=1,\cdots,k,\nn\\
&&\rho_s=\sigma_{s-k}(\xi_{k+1}^0,\cdots,\xi_l^0), \quad s=k+1,\cdots,l-2,\nn\\
&&\rho_{l-1}=\prod_{s=k+1}^l\left(e^{i\pi v_s^0}+e^{-i\pi v_s^0}\right),
\quad \rho_l=\prod_{s=k+1}^l\left(e^{i\pi v_s^0}-e^{-i\pi v_s^0}\right)\nn
\eeqa
with
\eqa
&&\xi^0_m=e^{2\pi iv_m^0}+e^{-2\pi iv_m^0},\nn\\
&&v_1^0=x_1^0,\ v_m^0=x^0_m-x^0_{m-1},\quad m=2,\cdots,l-2,\nn\\
&&v_{l-1}^0=x_l^0+x_{l-1}^0-x^0_{l-2},\quad v_{l}^0=x^0_{l-1}-x_{l}^0.\nn
\eeqa
With these we obtain
\beq\label{add1.28}
\det\left(\dfrac{\p y_i^0(\bx^0)}{\p \rho_j}\right)=\frac12 (\rho_k)^{l-k-1}.
\eeq

\noindent{iv)} For the case $D_l$ with $k=l-1$ we have
\eqa\label{add1.29}
&&y_m^0(\bx^0)=\rho_m,\quad m=1,\cdots,l-2,\nn\\
&&y_{l-1}^0(\bx^0)={\rho_l},\quad
y_l^0(\bx^0)=\frac{\rho_{l-1}}{{\rho_l}},\nn
\eeqa
where the functions $\rho_j$ are defined by
\beq\label{add1.30}
\rho_j=\sigma_j(e^{2\pi i v_1^0},\cdots,e^{2\pi i v_l^0}),\quad j=1,\cdots,l-1,\quad \rho_l=\prod_{s=1}^le^{\pi i v_s^0}\nn
\eeq
with
\eqa
&& v_1^0=x_1^0, \ v_m^0=x_{m}^0-x_{m-1}^0,\quad 2\le m\le l-2,\nn\\
&& v_{l-1}^0=x^0_{l}+x_{l-1}^0-x_{l-2}^0,\  v_{l}^0=x^0_{l-1}-x_{l}^0.\nn
\eeqa
With these we obtain
\beq
\det\left(\dfrac{\p y_i^0(\bx^0)}{\p \rho_j}\right)=-\frac1{ \rho_l}.
\eeq

\noindent{v)} For the case $D_l$ with $k=l$ the functions $y_j^0(\bx^0)$ and
$\rho_j$ are defined in the same way as in the above case iv), except
$$
y_{l-1}^0(\bx^0)=\frac{\rho_{l-1}}{ {\rho_l}},\quad
y_{l}^0(\bx^0)= {\rho_l},\quad
v^0_l=x_{l}^0-x_{l-1}^0.
$$
With these we obtain
\beq
\det\left(\dfrac{\p y_i^0(\bx^0)}{\p \rho_j}\right)=\frac1{\rho_l}.
\eeq

From the above calculation of the Jacobian $\det\left(\frac{\pal y^0_i({\bf x}^0)}{\pal \rho_j}\right)$
and from the algebraic independence of the functions $\rho_1,\dots,\rho_l$ we deduce the algebraic independence
of the functions $y_1^0(\bx^0)$,$\cdots$,$y_l^0(\bx^0)$.
This completes the proof of the theorem.
$\quad$\hfill \raisebox{0.11truecm}{\fbox{}}

\vskip0.4truecm

\section{Frobenius manifold structures on the orbit space of $\widetilde{W}^{(k)}(C_l)$}\label{sec-3}
\subsection{Flat pencils of metrics on the orbit space of ${\widetilde{W}}^{(k)}(C_l)$}

Let $\mathcal{M}$ be the orbit space defined as ${\rm Spec}{\mathcal
A}$ of the extended affine Weyl group $\widetilde{W}^{(k)}(C_l)$ for
any fixed $1\le k\le l$. Following \cite{DZ1998} we define an indefinite
metric $(~,~)^{\sptilde}$ on
\ian{
$\widetilde{V}_\mathbb{C} = \widetilde{V} \otimes_\mathbb{R} \mathbb{C}$
where}
$\widetilde V$ is the orthogonal direct sum of $V$ and $\mathbb
R$. Here ${V}$ is endowed with the $W$-invariant Euclidean metric
\beq (d x_s,
dx_n)^{\sptilde}=\frac{s}{4\pi^2},\quad 1\le s\le n\le l \label{DS3.1} \eeq
and $\mathbb R$ is endowed with the metric \beq
(dx_{l+1},dx_{l+1})^{\sptilde}=-\frac1{4 k
\pi^2}. \label{DS3.2} \eeq
The set of generators for the ring ${\mathcal A}={\mathcal
A}^{(k)}(C_l)$ are defined by (\ref{ip-a}), (\ref{ip-b}),
(\ref{add1.5}) with $\gamma=1$. They form a system of global coordinates on
${\mathcal M}$. We now introduce a system of
local coordinates on ${\mathcal M}$ as follows
\beq\label{zh9}
y^1=\ty_1,\dots, y^l=\ty_l,\ y^{l+1}=\log \ty_{l+1}=2 \pi i
x_{l+1}.
\eeq
\ian{They live on the universal covering $\widetilde{\mathcal M}$ of
${\mathcal M}\setminus\{\ty_{l+1}=0\}$. The projection
\beq
\Pr: \widetilde V\to \widetilde{\mathcal M}
\eeq}
induces a symmetric
bilinear form on $T^{*}{\mathcal{M}}$
\begin{equation}
(d y^i,d y^j)^{\sptilde}\equiv
g^{ij}(y):=\dsum_{a,b=1}^{l+1}\dfrac{\p y^i}{\p x_a}\dfrac{\p
y^j}{\p x_b}(dx_a,dx_b)^{\sptilde}.\label{bei-5}
\end{equation}

{\begin{prop} \label{zh-9}
The functions $g^{ij}(y)$ and $\Gamma^{ij}_n(y)$ defined by
\eqref{bei-5} and
\beq\label{zh-3}
\sum_{n=1}^{l+1} \Gamma_n^{ij}(y)dy^n=\sum_{p, q, r=1}^{l+1} \dfrac{\p y^i}{\p x_p}\dfrac{\p^2 y^j}{\p x_q \p x_r}(dx_p, dx_q)^\thicksim dx_r
\eeq
are weighted homogeneous
polynomials in ${y}^1,\cdots$,~$y^l$, $e^{y^{l+1}}$ of the degree
\eqa
&&\deg g^{ij}(y)=\deg y^i+\deg y^j,\\
&&\deg \Gamma^{ij}_n(y)=\deg y^i+\deg y^j-\deg y^n,
\eeqa
where $\deg y^j=d_j$ and $\deg y^{l+1}=d_{l+1}=0$.
\end{prop}
\pf The proposition follows from
Theorem \ref{thm1}.
\epf
From the above proposition we know that $\det(g^{ij})$ is a polynomial of
$y^1, \dots, y^l$ and $e^{y^{l+1}}$, so the discriminant $\Sigma$ of the extended affine Weyl group is an algebraic subvariety of $\mathcal M$. It was shown in \cite{DZ1998} that $\Sigma$ is the $\Pr$-image of the hyperplanes
\begin{equation}
\{(\bx,x_{l+1})|(\beta,\bx)=r\in\mathbb{Z},\ x_{l+1}=\text
{arbitrary}\}, \quad \beta\in \Phi^+,\label{add3.7}
\end{equation}
where $\Phi^+$ is the set of all positive roots.
The matrix $(g^{ij})$ is invertible on ${\mathcal M}\setminus \Sigma$ and the inverse matrix
$(g^{ij})^{-1}$ defines a flat metric on ${\mathcal M}\setminus \Sigma$.
The functions $\Gamma^{ij}_n(y)$  defined by \eqref{zh-3} coincide with the contravariant components
\[-\sum_{s=1}^{l+1} g^{is}(y) \Gamma^j_{sn}(y)\] of the
Levi-Civita connection of $(g^{ij})$ on ${\mathcal M}\setminus \Sigma$.
}

We now proceed to look
for other flat metrics on a certain subvariety in $\mathcal M$ that are compatible with the
metric $(g^{ij})^{-1}$.
To this end,
let us introduce the following new coordinates on ${\mathcal M}$:
\beq\label{zh-19}
\ta^j=\left\{\begin{array}{ll}e^{k y^{l+1}},&j=0,\\
y^j e^{(k-j) y^{l+1}}, &j=1,\cdots,k-1,\\
y^j,  &j=k,\cdots,l
\end{array}\right.
\eeq and denote
\beq \om_j=2\pi
i(x_j-x_{j-1}),\quad \om_{l+1}=y^{l+1}=2\pi ix_{l+1}, \quad j=1,\cdots,l.
\eeq
In the coordinates $\om_1,\dots, \om_{l+1}$ the indefinite metric on
$\widetilde V$ has the form
\beq \left((d \om_i,d \om_j)^{\sptilde} \right)=\hbox{diag}(-1,\dots,-1,\frac1{k}). \label{DS3.12}\eeq Define
\begin{equation}
P(u):=\dsum_{j=0}^{l} u^{l-j}\theta^j=e^{k \om_{l+1}}
\prod_{j=1}^{l}(u+\xi_j).\label{P1}
\end{equation}
We can easily verify that the function $P(u)$ satisfies \eqa
&&\dfrac{\p P(u)}{\p
\om_a}=\dfrac{1}{u+\xi_a}P(u)(e^{\om_a}-e^{-\om_a}),\quad
1\leq a \leq l;\label{eqp-a}\\
&&\dfrac{\p P(u)}{\p \om_{l+1}}=k P(u), \quad P'(u):=\dfrac{\p
P(u)}{\p u}=P(u)\dsum_{a=1}^{l}\frac{1}{u+\xi_a}.\label{eqp-b}
\eeqa

\begin{lem}\label{add3.2}
The following formulae hold true for the generating functions of the
metric $(g^{ij})$ and the contravariant components of its
Levi-Civita connection $\Gamma^{ij}_k$ in the coordinates $\ta^0,\dots,
\ta^{l}$:
\eqa
&&\dsum_{i,j=0}^{l}(d\ta^i,d\ta^j)^{\sptilde}u^{l-i}v^{l-j}=(d P(u), d P(v))^{\sptilde}\nn\\
&&
=(k-l)P(u)P(v)+\dfrac{u^2-4}{u-v}P'(u)P(v)-\dfrac{v^2-4}{u-v}P(u)P'(v),~~~~\label{gg}
\eeqa
\vskip -0.75truecm
\eqa && {\dsum_{i,j,r=0}^{l}\Gamma^{ij}_r(\ta)
d\ta^r u^{l-i}v^{l-j}=\dsum_{a,b,r=1}^{l+1}\dfrac{\p P(u)}{\p
\om_a}\dfrac{\p^2 P(v)}{\p \om_b \p
\om_r}d\om_r(d\om_a,d\om_b)^{\sptilde}}\nn\\
&&\quad =(k-l)P(u)dP(v)+\dfrac{u^2-4}{u-v}P'(u)dP(v)-\dfrac{v^2-4}{u-v}P(u)dP'(v)\nn\\
&&\qquad+\dfrac{uv-4}{(u-v)^2}P(v)dP(u)-\dfrac{uv-4}{(u-v)^2}P(u)dP(v).
\label{gamma}
\eeqa
 {Here
$\Gamma^{ij}_r(\theta)$ are the contravariant components
of the Levi-Civita connection of $(g^{ij})$ represented in the coordinates $\theta^0, \theta^1, \dots, \theta^l$.}
\end{lem}
\pf  By using \eqref{eqp-a} and \eqref{eqp-b}, we have
\begin{equation*}
\begin{array}{rl}
&(dP(u),dP(v))^{\sptilde}=\dfrac{1}{k}\dfrac{\p P(u)}{\p \om_{l+1}}\dfrac{\p P(v)}{\p
\om_{l+1}}-\dsum_{a=1}^l\dfrac{\p P(u)}{\p \om_a}\dfrac{\p
P(v)}{\p \om_a}\\
&=k P(u)P(v)-\dsum_{a=1}^l P(u)P(v)\frac{\xi_a^2-4}{(u+\xi_a)(v+\xi_a)}\\
&=k P(u) P(v) -\dsum_{s=1}^l P(u)P(v)\left(1-\frac{u^2-4}{u-v} \frac1{u+\xi_a}
+\frac{v^2-4}{u-v} \frac1{v+\xi_a}\right)\nn\\
&=(k-l)P(u)P(v)+\dfrac{u^2-4}{u-v}P'(u)P(v)-\dfrac{v^2-4}{u-v}P(u)P'(v).
\end{array}
\end{equation*}
So we proved the first formula, the second formula can be proved
in the same way. The lemma is proved.
\epf
The above lemma shows
that in the coordinates $\ta^0,\dots, \ta^{l}$ the functions
$g^{ij}(\ta)$ are quadratic polynomials, and the contravariant
components $\Gamma^{ij}_s$ are homogeneous linear
functions\footnote{These metrics give rise to a quadratic Poisson
structure on the space of ``loops" $\{S^1\to M\}$ (see \cite{Du2} for the details):
\beq
\{\theta^i(a), \theta^j(b)\}=g^{ij}(\theta(a))\delta'(a-b)+\Gamma_s^{ij}(\theta(a))
\theta_a^s\delta(a-b).\nn
\eeq
We plan to study such important  class of quadratic metrics and Poisson structures
in a separate publication.}. To find flat metrics that are compatible with this quadratic
metric $g^{ij}(\ta)$, we need the following lemma.
\begin{lem}\label{lem-du}
If there is a set of constants $\{c_0,\dots, c_l\}$ such that

(i)\, the functions
\eqa
&&g^{ij}(\theta^0+c_0\lm,\theta^1+c_1\lm,\dots,\theta^l+c_l\lm),\nn\\
&&\Gamma^{ij}_s(\theta^0+c_0\lm,\theta^1+c_1\lm,\dots,\theta^l+c_l\lm)\nn
\eeqa are linear in the parameter $\lm$ for $1\le i,j,s\le l+1$, and

(ii)\, the matrix $(\eta^{ij})$ with \beq \eta^{ij}={\mathcal L}_{e}
g^{ij},\quad e=\sum_{j=0}^l c_j \frac{\pal}{\pal \theta^j} \label{DS3.18}\eeq is
nondegenerate on certain open subset ${\mathcal U}$ of ${\mathcal M}$.

Then the metrics $(g^{ij}), (\eta^{ij})$  form a flat pencil, i.e.,
the linear combination $(g^{ij}+\lm \eta^{ij})$ yields a flat metric on ${\mathcal U}$
for any $\lm$ satisfying $\det(g^{ij}+\lm \eta^{ij})\ne 0$, and the
contravariant components of the Levi-Civita connection for this
metric equal to
\beq \Gamma^{ij}_s+\lm\, \gamma^{ij}_s.
\eeq
Here
$\gamma^{ij}_s$ are the contravariant components of the Levi-Civita
connection for the metric $(\eta^{ij})$ which can be evaluated by
$\gamma^{ij}_s={\mathcal L}_e \Gamma^{ij}_s$.
\end{lem}
\pf For the proof of this lemma, see Appendix D of \cite{Du1}.\epf

\begin{thm}\label{thm-du}
For any fixed integer $0\le m\le l-k$ there is a flat pencil of metrics $(g^{ij}), (\eta^{ij})$
on a certain open subset ${\mathcal U}$ of ${\mathcal M}$
with $(g^{ij})$ given by (\ref{bei-5}) and $\eta^{ij}={\mathcal L}_e g^{ij}$. Here the vector
field $e$ has the form
\beq\label{unity}
e:=\sum_{j=k}^{l} c_j \frac{\pal}{\pal \theta^j}=\sum_{j=k}^{l} c_j \frac{\pal}{\pal y^j}
\eeq
 {with the constants $c_k,\cdots, c_l$ defined by the generating function}
\beq
P_0(u)=\sum_{j=k}^l c_j u^{l-j} =(u+2)^m (u-2)^{l-k-m}.
\eeq
Explicitly,  $c_j=(-2)^{j-k}\dsum_{s=0}^{m}(-1)^{m-s}\binom{m}{s} \binom{l-k-m}{l-j-s}$ for $j=k,\cdots, l$.
\end{thm}
\pf Firstly we want to find the constants $c_0,\dots, c_l$
satisfying the condition $(i)$ in Lemma \ref{lem-du}. It suffices to find
a polynomial $P_0(u)=\dsum_{j=0}^l c_j u^{l-j}$ such that after the
shift
\beq P(u)\mapsto P(u)+\lm P_0(u),\quad P(v)\mapsto P(v)+\lm
P_0(v), \nn \eeq
the right hand side of \eqref{gg} and \eqref{gamma} are linear
in $\lm$. This yields that $P_0(u)$ and $P_0(v)$ must satisfy
\beq (k-l)P_0(u)P_0(v)+\dfrac{u^2-4}{u-v}P_0'(u)P_0(v)-\dfrac{v^2-4}{u-v}P_0(u)P_0'(v)=0.
\label{eq3.20} \eeq
Separating the variables and integrating one obtains
$$
P_0(u) = a \left( \frac{u-2}{u+2}\right)^b \left[(u-2)(u+2)\right]^{\frac{l-k}2}= (u-2)^{\frac{l-k}2+b}(u+2)^{\frac{l-k}2 -b}
$$
for some constants $a$, $b$. This is a polynomial \emph{if and only if} $m:=\frac{l-k}2 -b$ is a non-negative integer.
Hence any polynomial solution to eq. \eqref{eq3.20} must have the form $P_0(u)=a(u+2)^m (u-2)^{l-k-m}$
for an integer where $0\le m\le l-k$.  Thus, up to a common factor,
the constants $c_0,\dots, c_l$ are determined by
\beq\sum_{j=0}^l c_j u^{l-j}=(u+2)^m (u-2)^{l-k-m}.\nn\eeq
Actually, by comparing the degrees of $u$, we know $c_j=0$ for $j=0,\cdots,k-1$.

Next we want to check the condition $(ii)$ in Lemma \ref{lem-du}. In order to do this, taking any fixed integer
$0\le m\le l-k$ we consider the following linear change of coordinates
$$
(y^1,\dots,y^{l+1})\mapsto (\tau^1,\dots,\tau^{l+1})
$$
defined by the relations $\tau^{l+1}=y^{l+1}$ and
 {\eqa
&&\sum_{j=0}^l \theta^j u^{l-j}=\sum_{j=0}^{l-m}\varpi^{{j}}
\left( u+2 \right) ^{m} \left( u-2 \right)^{l-m-j}\nn\\
&&\qquad \qquad\quad -\sum _{j=l-m+1}^{l}\varpi^{{j}} \left( u+2 \right) ^{l -j}
\left( u-2 \right) ^{j-k-1},\label{yz}
\eeqa}
where
\beq \varpi^j=\left\{\begin{array}{ll}e^{k\, \tau^{l+1}},&j=0,\\
\tau^j e^{(k-j) \tau^{l+1}}, &j=1,\cdots,k-1,\\
\tau^j,  &j=k,\cdots,l.
\end{array}\right. \label{add3.24}
\eeq
Then,
$$\sum_{j=0}^l\frac{\p \theta^j}{\p \tau^k}u^{l-j}=(u+2)^m(u-2)^{l-k-m}=\sum_{j=0}^l c_j u^{l-j}.$$
This means that in terms of the new coordinates $\tau^i$ the vector field $e$ defined in (\ref{unity})
has the expression
\beq e=\sum_{j=0}^l\frac{\p \theta^j}{\p \tau^k}\frac{\p}{\p \theta^j}
=\frac{\p}{\p \tau^k}.\nn
\eeq
Furthermore, observe that the left hand side of (\ref{yz}) coincides with the polynomial $P(u)$.
\ian{By substituting the expressions of  $P(u), P(v)$ given by the right
\ian{hand side of (\ref{yz}) into both sides of \eqref{gg}, we get }
an identity which on the right hand side has at most linear terms in $\tau^k$ due
to the definition of $P_0(u)$ and $P_0(v)$. Differentiating both sides by $\tau^k$ and dividing by $P_0(u)$ and $P_0(v)$, we obtain an identity \ian{relating two rational functions} in $u$ ($v$ is kept
as a parameter) with poles at $u = 2$ and $u =-2$. Comparing the regular parts
and the polar parts at $u = 2$ and at $u = -2$ we get explicit formulae for $\eta(d \varpi^i,d \varpi^j)$.
Finally, with the use of \eqref{add3.24} we obtain explicit formulae for the
matrix $({\eta}^{ij}(\tau))$ with entries }
\beq
{\eta}^{ij}(\tau)=\mathcal{L}_eg^{ij}(\tau)
\eeq
which has the block form
\beq\label{bdf-1}
\left(\begin{array}{ccccccccccccc}
~&W_1&\begin{array}{c}P_1\\ \vdots\\P_{k-1}\end{array}~&~&~&\\
P_1&\cdots P_{k-1}&P_k&0&0&1\\
~&~&0&W_2&0&0\\
~&~&0&0&W_3&0\\
~&~&1&0&0&0\\
\end{array}\right),
\eeq
where $W_i$ are triangular blocks
\beq\label{bdf-2}
W_1=\left(\begin{array}{cccccccc}
0&0&0&\cdots&0& k\\
0&0&0&\cdots&k&R_{1}\\
0&0&0&\cdots&R_1&R_{2}\\
\vdots&\vdots&\vdots&\vdots&\ddots&\vdots\\
k&R_{1}&R_2&\cdots&{}&R_{k-2}
\end{array}\right),\\
\eeq
\beq
W_2=\left(\begin{array}{cccc}
Q_1&Q_2&\cdots&Q_{l-k-m}\\
Q_2&Q_3&\cdots&0\\
\vdots&\vdots&\vdots&0\\
Q_{l-k-m}&0&\cdots&0
\end{array}\right),\quad
W_3=\left(\begin{array}{cccc}
S_1&S_2&\cdots&S_m\\
S_2&S_3&\cdots&0\\
\vdots&\vdots&\vdots&0\\
S_m&0&\cdots&0
\end{array}\right)\label{bdf-3}
\eeq
with entries
\eqa
&&R_{j}=4(k-j+1) \tau^{j-1} e^{\tau^{l+1}} +(k-j) \tau^j,\quad \tau^0=1,\nn\\
&&P_j=4(k-j+1) \tau^{j-1} e^{\tau^{l+1}},\nn\\
&&Q_s=4 s \tau^{k+s}+(1-\delta_{s,l-k-m}) (s+1) \tau^{k+s+1},\nn\\
&&S_r=4r\tau^{l-m+r}-4(1-\delta_{r,m}) r\tau^{l-m+r+1},\\
&&1\le j\le k,\quad 1\le r\le m,\quad 1\le s\le l-k-m.\nn
\eeqa
A simple computation gives
\beq\label{zh-4}
\det (\eta^{ij})=(-1)^l k^{k-1} 4^{l-k}m^m(l-k-m)^{l-k-m} (\tau^{l-m})^{l-k-m}{(\tau^l)}^m.
\eeq
\ian{So the metrix
$(\eta^{ij}(\tau))$ does not degenerate on ${\mathcal M}\setminus\{\tau^l=0\}\cup\{\tau^{l-m}=0\}$ when $m\ne 0, l-k$, and on ${\mathcal M}\setminus\{\tau^l=0\}$ when $m=0$ or $m=l-k$.} This completes the proof of the theorem.\epf

\begin{rem}\label{rem-1}
(1). The block $W_2$ or $W_3$ does not appears in the matrix (\ref{bdf-1}) when $m=l-k$ or $m=0$ (e.g.\cite{DZZ2005}).
(2). The flat pencil of metrics that corresponds to a fixed integer $m$ is equivalent to the one
that corresponds to the integer $l-k-m$, this is due to the fact that under replacement
$u\mapsto -u$ the polynomial $P_0(u)=(u+2)^m (u-2)^{l-k-m}$ is transformed to the
polynomial $(-1)^{l-k} (u+2)^{l-k-m} (u-2)^m$.
\end{rem}

\ian{\begin{rem}\label{zh-11}
 From \eqref{zh-4} it follows that the zero loci of $\det(\eta^{ij}(\tau))$ are given by that of $\tau^l, \tau^{l-m}$ when $m\ne 0, l-k$, and
by that of $\tau^l$ when $m=0$ or $m=l-k$.
By using \eqref{P1}, \eqref{add3.24} and \eqref{yz} we know that
\beq 4^m \tau^{l-m} =P(2)=4^l \widetilde{y}_{l+1}^k  \prod_{j=1}^l \cos^2\left(\pi(x_j-x_{j-1})\right)=4^l f_1
\nn\eeq
when $m\ne l-k$, and
\beq
-(-4)^{l-k-1}\tau^{l} =P(-2)=(-4)^l \widetilde{y}_{l+1}^k  \prod_{j=1}^l \sin^2\left(\pi(x_j-x_{j-1})\right)=(-4)^l f_2\nn
\eeq
when $m\ne 0$, here $f_1, f_2$ are defined in \eqref{zh-12}. We note that
$\tau^l=4^l f_1$ when $m=0$, and $\tau^l=-(-4)^{k+1} f_2$ when $m=l-k$.
\end{rem}}

\begin{cor}\label{zh-15}
In the coordinates $\tau^1,\dots, \tau^{l+1}$ the components $g^{ij}(\tau)$,
$\Gamma_m^{ij}(\tau)$ of the metric (\ref{bei-5}) and its Levi-Civita connection
are weighted homogeneous polynomials with degrees
\begin{equation}
\deg g^{ij}=d_i+d_j,\quad \deg \Gamma_s^{ij}(\tau)=d_i+d_j-d_s.
\end{equation}
They are at most linear in $\tau^k$. \ian{Here  $\deg \tau^j=d_j$ and $\deg \tau^{l+1}=d_{l+1}=0$.}
\end{cor}


\subsection{Flat coordinates of the metric $(\eta^{ij})$}

In this subsection, we will show that the flat coordinates of the metric
$(\eta^{ij})$ defined in the last subsection are algebraic functions of
$\tau^1,\dots, \tau^{l+1}, e^{\tau^{l+1}}$. To this end, we first perform changes of coordinates to simplify
the matrix $(\eta^{ij}(\tau))$.

\begin{lem}\label{lem3.1}
There exists a system of coordinates $z^1,\dots, z^{l+1}$ of
the form
\eqa &&z^j=\tau^j+p_j(\tau^1,\dots,\tau^{j-1},e^{\tau^{l+1}}), ~~1\leq j \leq k,
\label{zh2-a}\eeqa
\eqa
&& z^j=\tau^j+\sum_{s=j+1}^{l-m} c^j_s \,\tau^s,\quad k+1\le j\le
l-m-1,\label{zh2-b}\eeqa
\eqa
&& z^j=\tau^j+\sum_{s=j+1}^{l} h^j_s\, \tau^s,\quad l-m+1\le j\le
l-1,\label{zh2-c}\\
&& z^{l-m}=\tau^{l-m},\quad z^l=\tau^l,\quad z^{l+1}=\tau^{l+1},\label{zh2-d}
\eeqa
where $c^j_s$ and $h^j_s$ are some constants and $p_j$ are homogeneous polynomials of
degree $d_j$ such that in the
new coordinates $z^i$  the components of the metric $(\eta^{ij})$ can still been
encoded into a block diagonal matrix of the form (\ref{bdf-1})--(\ref{bdf-3})
with the entries replaced by
\eqa && R_{j}=0,\quad
P_j=0,\quad Q_s=4 s z^{k+s},S_r=4r z^{l-m+r},\label{zh3}\\
&& 1\le j\le k,\quad 1\le s\le
l-k-m,\quad 1\le r\le m. \nn
\eeqa
\end{lem}
\pf Let us first note that the $(k+1)\times (k+1)$ matrix $({\tilde{\eta}}^{ij})$
which has entries
\beq
{\tilde {\eta}}^{ij}=\eta^{ij}(\tau),\ {\tilde{\eta}}^{k+1,m}={\tilde{\eta}}^{m,k+1}=\delta_{m,k},\quad
1\le i,j\le k,\ 1\le m\le k+1
\eeq
coincides, under renaming of the label of coordinate $\tau^{l+1}\mapsto \tau^{k+1}$,
with the  matrix $(\eta^{ij}(\tau))_{(k+1)\times (k+1)}$ that was constructed in the last subsection
with respect to the extended affine Weyl group ${\widetilde
W}^{(k)}(C_k)$. Thus by using the results of \cite{DZ1998} we can
find homogeneous polynomials $p_j, 1\le j\le k$ such that under
the change of coordinates \eqref{zh2-a} and $z^j=\tau^j$, $k+1\le j\le
l+1$ the matrix $({\eta}^{ij}(z))$ has the form (\ref{bdf-1})--(\ref{bdf-3}) with
entries
\eqa
&&R_{j}=0,\quad P_j=0,\quad Q_s=4 s z^{k+s}+(1-\delta_{s,l-k-m}) (s+1) z^{k+s+1},\nn \\
&& S_r=4r z^{l-m+r}-4 (1-\delta_{m,r}) r z^{l-m+r+1},\nn \\
&&1\le j\le k,\quad 1\le r\le m,\quad 1\le s\le l-k-m.\nn
\eeqa

To finish the proof of the lemma, we need to perform a second change of coordinates.
To this end, denote by $\Psi$ an $n\times n$ matrix with entries
 as linear functions of $a^1,\dots, a^n$
\eqa
&&\psi^{ij}(a)=4(i+j-1)a^{i+j-1}+\kappa(i,j)a^{i+j},\quad i,j\ge 1, \label{TR1-a}\\
&& \kappa(i,j)=i+j, \quad or \quad -4(i+j-1).\label{TR1-b}
\eeqa
Here $a^s=0 \ {\rm for}\  s\ge n+1$.
We require a linear transformation of the triangular form
\beq
a^j=\dsum_{\al=j}^n \,B_{\al}^jb^{\al},\quad B^j_j=1,~~j\ge 1
\eeq
such that
\beq\label{TR6}
\dsum_{r,s=1}^n 4(r+s-1)b^{r+s-1}\dfrac{\p a^i}{\p b^r}\dfrac{\p
a^j}{\p b^s}=\psi^{ij}(a),
\eeq
\ian{where $b^s=0 \ {\rm for}\  s\ge n+1$.}
Equivalently, the constants $B^i_j$ must satisfy the relations
\eqa
&&4(i+j-1)B^{i+j-1}_\gamma+\kappa(i,j)B^{i+j}_\gamma=4 \gamma
\dsum_{\alpha+\beta=\gamma+1}B^i_\alpha B^j_\beta,\nn\\
&&\quad i+j\le \gamma\le n.\label{TR8}
\eeqa
Consider the generating functions
\beq\label{TR9}
f^i(t)=\sum_{\alpha\ge0}B^i_{i+\alpha} t^\alpha,\quad i=1,2,\dots.
\eeq
Then the relations in \eqref{TR8} can be encoded into the following equations:
\beq\label{TR10}
4(i+j-1)t^{i+j-2}f^{i+j-1}+\kappa(i,j)
t^{i+j-1}f^{i+j}= 4\frac{d}{dt}\left(t^{i+j-1}f^if^j\right).
\eeq
When $\kappa(i,j)=i+j$ and $\kappa(i,j)=-4(i+j-1)$, this system of equations has, respectively, the following
solutions:
\beq\label{TR11}
f^i(t)=\cosh\left(\frac{\sqrt{t}}2\right)
\left(\frac{2\sinh\left(\frac{\sqrt{t}}2\right)}{\sqrt{t}}\right)^{2i-1},
\eeq
and
\beq\label{TR111}
f^i(t)=\left(\frac{\tanh(\sqrt{t})}{\sqrt{t}}\right)^{2i-1}.
\eeq
From the above result we know the existence of constants $c^j_s$ and $h_s^j$
such that under the change of coordinates
\eqa
&&z^i\mapsto z^i,\ i=1,\dots, k,l-m,l, l+1,\nn\\
&&z^j\mapsto z^j+\sum_{s=j+1}^{l-m} c^j_s\, z^s,\quad k+1\le j\le
l-m-1,\nn\\
&& z^j\mapsto z^j+\sum_{s=j+1}^{l} h^j_s\, z^s,\quad l-m+1\le j\le
l-1,\nn
\eeqa
the matrix $({\eta}^{ij}(z))$ has the form (\ref{bdf-1})--(\ref{bdf-3}) and with entries given by (\ref{zh3}).
The lemma is proved. \epf

\begin{lem}
Under the change of coordinates
\eqa
&&w^i=z^i, \quad i=1,\dots,k,\ l+1,\label{w1}\\
&&w^{k+1}=z^{k+1} (z^{l-m})^{-\frac1{2 (l-m-k)}},\label{w2}\\
&&
w^s=z^s(z^{l-m})^{-\frac{s-k}{l-m-k}},
\ s=k+2,\cdots,l-m-1,
\label{w3}\\
&& w^{l-m}=(z^{l-m})^{\frac{1}{2(l-m-k)}},\label{w4}\\
&&w^{l-m+1}=z^{l-m+1} (z^l)^{-\frac1{2m}},\label{w5}\\
&&w^r=z^r(z^{l})^{-\frac{r+m-l}{m}},\
r=l-m+2,\cdots,l-1,\label{w6}\\
&&w^{l}=(z^{l})^{\frac{1}{2m}},\label{w7}
\eeqa
the components of the metric $(\eta^{ij}(z))$ are transformed to the form
\beq\label{mw}
\left(\begin{array}{ccccc}
A&0&0&0&0\\
0&0&0&0&1\\
0&0&B_1&0&0\\
0&0&0&B_2&0\\
0&1&0&0&0
\end{array}\right),
\eeq
where the matrix $A=A_{(k-1)\times (k-1)}$ has entries $A^{ij}=\delta_{i,k-j} k$ and
the upper triangular matrices $B_1$ and $B_2$ have the form
\beq\label{b1}
B_1=\left(\begin{array}{cccccccccccc}
0&0& 0&0&\cdots &0 &2& \\
0&H_{k+3}&H_{k+4}&\cdots&H_{l-m-1}&H_{l-m}&&\\
0&H_{k+4}&H_{k+5}&\cdots &H_{l-m}&&\\
\vdots&\vdots &\vdots&&&&&&\\
0&H_{l-m}&& &&&&\\
2&&&&&&&
\end{array}\right)\eeq
and
\beq\label{b2}
B_2=\left(\begin{array}{cccccccccccc}
0&0& 0&0&\cdots &0 &2& \\
0&H_{l-m+3}&H_{l-m+4}&\cdots&H_{l-1}&H_{l}&&\\
0&H_{l-m+4}&H_{l-m+5}&\cdots &H_{l}&&\\
\vdots&\vdots &\vdots&&&&&&\\
0&H_{l}&& &&&&\\
2&&&&&&&
\end{array}\right)\eeq
with
\eqa
&&H_{k+s}=4 s (w^{l-m})^{-2} w^{k+s},~~ H_{l-m}=4 (l-m-k) {(w^{l-m})}^{-2},\nn \\
&&H_{l-m+j}=4 j (w^{l})^{-2} w^{l-m+j},~~ H_{l}=4m {(w^{l})}^{-2}, \label{b3}\\
&& 3\le s\le l-m-k-1,\quad 3\le j\le m-1.\nn
\eeqa
\end{lem}
\pf By a straightforward calculation. \epf

\begin{rem}
When $m=l-k$, the matrix $B_1$ does not appear in (\ref{mw}), i.e., the matrix given in (\ref{mw}) has the form
\beq
\left(\begin{array}{ccccc}
A&0&0&0\\
0&0&0&1\\
0&0&B_2&0\\
0&1&0&0
\end{array}\right),\nn
\eeq
In this case we use the formulae (\ref{w1}), (\ref{w5})--(\ref{w7}) for the change of coordinates.
When $m=l-k-1$, we have $B_1=1$, and we use the formulae (\ref{w1}), (\ref{w4})--(\ref{w7}) to define the new
coordinates. When $m=l-k-2$, the matrix $B_1$ has the form
$
\left(\begin{array}{cc} 0&2\\2&0
\end{array}\right).
$
We understand the above lemma in a similar way as we did for the cases when $m=0, 1, 2$.
\end{rem}

\begin{thm}\label{DZZ-thm4.4}
We can choose the flat coordinates of the metric $(\eta^{ij}(w))$
in the form
\eqa &&t^1=w^1,\dots, t^k=w^k,\
t^{l+1}=w^{l+1}, \nn \\
&& t^{k+1}=w^{k+1}+w^{l-m}\,h_{k+1}(w^{k+2},\dots,w^{l-m-1}),\nn \\
&& t^j=w^{l-m} (w^j+h_j(w^{j+1},\dots,w^{l-m-1})),\ k+2\le j\le l-m-1, \nn\\
&& t^{l-m}=w^{l-m},\nn\\
&& t^{l-m+1}=w^{l-m+1}+w^l\,h_{l-m+1}(w^{l-m+2},\dots,w^{l-1}),\nn \\
&& t^s=w^l (w^s+h_s(w^{s+1},\dots,w^{l-1})),\ l-m+2\le s\le l-1, \nn\\
&& t^l=w^l.\nn
\eeqa
Here $h_{l-m-1}=h_{l-1}=0$, $h_j$ are weighted homogeneous
polynomials of degree $\frac{k\,(l-m-j)}{l-m-k}$  for $j=k+1,\dots,l-m-2$
and $h_s$ are weighted homogeneous polynomials of degree $\frac{k\,(l-s)}{m}$
for $s=l-m+2,\dots,l-1$. The degrees of the coordinates $w^i$ are defined in a natural
way through the degrees of $y^i$ given in (\ref{add1.4}).
\end{thm}
\pf
From the block diagonal form (\ref{mw}) of the matrix $(\eta^{ij}(w))$ and the definition
(\ref{b1})--(\ref{b3}) of its entries,
we know that the flat coordinates can be chosen to have the form
\eqa
&&t^i=w^i, \ 1\le i\le k,\ i=l+1,\\
&&t^j=t^j(w^{k+1},\dots, w^{l-m}),\quad k+1\le j\le l-m\label{ft-1}\\
&&t^s=t^s(w^{l-m+1},\dots, w^{l}),\quad l-m+1\le s\le l.
\eeqa
Since the matrices $B_1$ and $B_2$ have the same form, and $B_1$ becomes constant when
$m=l-k$ or $m=l-k-1$, we only need to consider the flat coordinates (\ref{ft-1}) for the metric that
corresponds to the matrix $B_1$ defined in (\ref{b1}) with $m\le l-k-3$.

The functions $t^j=t^j(w^{k+1},\dots, w^{l-m})$
must satisfy the
 following system of PDEs
\beq\label{zh5}
\frac{\pal^2 t}{\pal w^a\pal w^b}-\sum_{c=k+1}^{l-m}\gamma^c_{ab}
   \frac{\pal t}{\pal w^c}=0,\quad a,b=k+1,\dots,l-m,
\eeq
\ian{where $\gamma^{c}_{ab}$ are the Christoffel symbols with respect to the metric $B_1$.
Let us introduce the $(l-m-k)\times (l-m-k)$ matrix
\beq
\Phi=(\phi_j^i),\quad {\phi}^i_j=\frac{\pal t^{k+i}}{\pal w^{k+j}},\quad 1\le i,j\le l-m-k,\nn
\eeq
where in the notation $\phi_j^i$ the upper (resp. lower) index denotes the row (resp. column) number of $\Phi$.}
Then the system (\ref{zh5}) can be written in the form
\beq
\pal_s\Phi=\Phi A_s,\quad \pal_s=\frac{\pal}{\pal w^s}, \quad s=k+1,\dots,l-m,\label{zh6-a}
\eeq
where the entries of the coefficient matrices $A_s$ are rational functions of $w^{k+1},\dots, w^{l-m}$.
It follows from the simple expressions of the entries of the matrix $B_1$ that the systems (\ref{zh6-a})
are regular at ${\bf w}=$$(w^{k+1}$, $\dots,w^{l-m})=0$
except for case when $s=l-m$, in this case the coefficient matrix has the form
\beq
A_{l-m}={\mbox{diag}}(0,\frac1{w^{l-m}},\dots,\frac1{w^{l-m}},0).\nn
\eeq
Note for all the cases with $m=k+1,\dots,l-m-1$ the entries of the matrices $A_s$ are
weighted homogeneous polynomials of $w^{k+1},\dots, w^{l-m}$.

On writing$\Phi$ in the form
\beq
\Phi=\Psi \,{\mbox{diag}}(1,w^{l-m},\dots,w^{l-m},1),\nn
\eeq
the systems in (\ref{zh6-a}) are converted to
\beq
\pal_s\Psi=\Psi B_s,\quad \pal_{l-m}\Psi=0,\quad s=k+1,\dots,l-m-1.\nn
\eeq
The entries of the coefficient matrices $B_s$ are now weighted homogeneous polynomials
of $w^{k+1},\dots, w^{l-m}$, thus we can find a unique solution $\Psi$ of the above systems
such that it is
analytic at ${\bf w}=0$ and
\beq
\left.\Psi\right|_{{\bf w}=0}={\mbox{diag}}(1,\dots,1).\nn
\eeq
From the weighted homogeneity of the coefficient matrices $B_s$ it follows that
the elements of $\Psi$ are also weighted homogeneous. Since $\deg w^j>0$ for $j=k+1,\dots, l-m$
we know that they are in fact polynomials of $w^{k+1},\dots, w^{l-m}$, and thus the results of the
theorem follow. The theorem is proved.\epf

Due to the above construction, we can associate the following
natural degrees to the flat coordinates
\eqa
&&\tilde d_j=\deg t^j:=\frac{j}{k},\quad 1\le j\le k,\label{zh18a}\\
&& \tilde d_s=\deg t^s:=\dfrac{2l-2m-2s+1}{2(l-m-k)},\quad  k+1\le s\le l-m,\label{zh18b}\\
&& \tilde d_\alpha=\deg t^\alpha:=\dfrac{2l-2\alpha+1}{2m},\quad  l-m+1\le \alpha\le l,\label{zh18c}\\
&&\tilde d_{l+1}=\deg t^{l+1}:=0,\quad \deg e^{t^{l+1}}:=\frac{1}{k},\label{zh18d}
\eeqa
and we readily have the
following corollary.

\begin{cor}\label{cor3.5}
In the flat coordinates $t^1,\dots, t^{l+1}$, the nonzero entries of
the matrix $(\eta^{ij}(t))$ are given by
\beq \label{zh7}
\eta^{ij}=\left\{\begin{array}{lll} k,\quad & j=k-i,\ & 1\le i\le k-1,\\
1,&i=l+1, j=k\ & {\mbox or}\ i=k,\ j=l+1,\\
4 (l-m-k),\quad  &j=l-m+k-i+1,\ &k+2\le i\le l-m-1,\\
2,&i=l-m, j=k+1\ &{\mbox or}\ i=k+1,\ j=l-m,\\
4m,\quad  &j=2l-m-i+1,\ &l-m+2\le i\le l-1,\\
2,&i=l, j=l-m+1\ &{\mbox or}\ i=l-m+1,\ j=l.
\end{array}\right.\eeq
 The entries of the matrix $(g^{ij}(t))$ and the
Christoffel symbols $\Gamma^{ij}_m(t)$ are weighted homogeneous
polynomials of $t^1,\dots,t^l, \dfrac{1}{t^{l-m}},\dfrac{1}{t^l}, e^{t^{l+1}}$ of
degrees $\td_i+\td_j$ and $\td_i+\td_j-\td_m$ respectively.
In particular,
\beq \label{zh8} \begin{array}{ll}g^{s, l+1}=\td_s
t^s, \quad & 1\leq s
\leq l,\quad g^{l+1,  l+1}=\dfrac{1}{k},\\
\Gamma_{j}^{l+1,i}=\td_j \delta_{i,j},\quad & 1\leq i, j \leq l+1.
\end{array}
\eeq
\end{cor}

The numbers $\td_1,\dots,\td_{l+1}$ satisfy a duality relation that is
similar to that of \cite{DZ1998}. To describe this duality
relation, let us delete the $k$-th vertex of the Dynkin diagram
${\mathcal R}$. We then obtain two components ${\mathcal R}\setminus
{\al_k}={\mathcal R}_1\cup {\mathcal R}_2$. For any given integer $0\le m \le l-k$, we
denote ${\mathcal R}_2={\mathcal R}_{21}\cup {\mathcal R}_{22}$, where
 ${\mathcal R}_{21}=\{\alpha_{k+1},\cdots,\alpha_{l-m}\}$ and
 ${\mathcal R}_{22}=\{\alpha_{l-m+1},\cdots,\alpha_{l}\}$.
On each component we have an involution $i\mapsto i^*$ given by the
reflection with
respect to the center of the component. Define
\beq\label{zh-n1}
k^*=l+1,\quad (l+1)^*=k,
\eeq
then we have
\beq\label{zh-n2}
\td_i+\td_{i^*}=1,\quad i=1,\dots, l+1,
\eeq
and from the above
corollary we see that $\eta^{ij}$ is a nonzero constant iff
$j=i^*$.
\subsection{Frobenius manifold structures on the orbit space of $\widetilde{W}^{(k)}(C_l)$}

Now we are ready to describe the Frobenius manifold structures on the orbit space of
the extended affine Weyl group ${\widetilde W}^{(k)}(C_l)$. Let us first recall the definition of Frobenius
manifold, see \cite{Du1} for details.

\begin{de} A {\it Frobenius algebra} is a pair $(A, <~,~>)$ where $A$
is a commutative associative algebra with a unity $e$ over a field $\mathcal{K}$ (in our case
$\mathcal{K}={\mathbb C}$) and $<~,~>$ is a $\mathcal{K}$-bilinear
symmetric nondegenerate {\it invariant} form on $A$, i.e.,
$$
<x\cdot y, z> = <x, y\cdot z>,\quad \forall\ x, y, z \in A.
$$

\end{de}
\begin{de}A Frobenius  structure of charge $d$
on an n-dimensional manifold $M$ is a structure of Frobenius algebra on the tangent spaces $T_tM=(A_t,<~,~>_t)$
depending (smoothly, analytically etc.) on the point $t$. This structure satisfies the following axioms:
 \begin{itemize}
\item[FM1.] The metric $<~,~>_t$ on $M$ is flat, and the unity vector field $e$ is covariantly constant, i.e.,
$\nabla e=0$. Here we
denote $\nabla$ the Levi-Civita connection for this flat metric.
\item[FM2.] Let $c$ be the 3-tensor $c(x,y,z):=<x\cdot y, z>$, $x,\, y,\,
z\in T_tM$. Then the 4-tensor $(\nabla_w c)(x,y,z)$ is symmetric in
$x,\, y,\, z, \, w \in T_tM$.
\item[FM3.] The existence on $M$ of a vector field $E$, called the Euler vector field,
which satisfies the conditions $\nabla\nabla E=0$ and
$$
[E, x\cdot y] -[E,x]\cdot y -x\cdot [E,y] = x\cdot y,
$$
$$
E<x,y>-<[E,x],y>-<x,[E,y]>=(2-d)<x,y>
$$
for any vector fields $x, y$ on $M$.
\end{itemize}
A manifold $M$ equipped with a Frobenius structure on it is called a Frobenius manifold.
 \end{de}

Let us choose local flat coordinates $t^1,\cdots t^n$ for the
invariant flat metric, then locally there exists a function
$F(t^1,\cdots,t^n)$, called the {\em potential} of the Frobenius
manifold, such that \beq < u\cdot v,w>=u^i v^j w^s \frac{\p^3
F}{{\p t^i}{\p t^j}{\p t^s}} \label{WDVV0} \eeq for any three
vector fields $u=u^i\frac{\p}{\p t^i}$, $v=v^j\frac{\p}{\p t^j}$,
$w=w^s\frac{\p}{\p t^{s}}$. Here and in what follows summations
over repeated indices are assumed. By definition, we can also
choose the coordinates $t^1$ such that $ e=\frac{\pal}{\pal
t^1} $. Then in the flat coordinates the components of of the
flat metric $<\frac{\p}{\p t^i},\frac{\p}{\p t^j}>$ can be
expressed in the form
\beq \frac{\p^3 F}{{\p t^1}{\p t^i}{\p
t^j}}=\eta_{ij},\quad i,j=1,\dots, n. \label{WDVV1} \eeq
The associativity of the Frobenius algebras is equivalent to the
following overdetermined system of equations for the function $F$
\beq\label{WDVV2} \frac{\p^3 F}{{\p t^i}{\p t^j}{\p
t^\lambda}}\eta^{\lambda\mu} \frac{\p^3 F}{{\p t^\mu}{\p t^k}{\p
t^m}}=\frac{\p^3 F}{{\p t^k}{\p t^j}{\p
t^\lambda}}\eta^{\lambda\mu} \frac{\p^3 F}{{\p t^\mu}{\p t^i}{\p
t^m}} \eeq for arbitrary indices $i,j,k,m$ from $1$ to $n$.

\ian{We assume that flat coordinates have been chosen so that}
the Euler vector field $E$ has the form
\beq E=\sum_{i=1}^n (\hat{d}_i t^i+r_i)\frac{\pal}{\pal t^i} \eeq
for some constants $\hat{d}_i, r_i,\, i=1,\dots,n$ which satisfy $
\hat{d}_1=1, r_1=0 $. From the axiom FM3, it follows that the
potential $F$ satisfies the quasi-homogeneity condition \beq
\label{WDVV3} \mathcal{L}_EF=(3-d)F+\mbox{quadratic polynomial in
t}. \eeq The system \eqref{WDVV1}--\eqref{WDVV3} is called the
{\sl $WDVV$ equations of associativity} which is equivalent to the
above definition of Frobenius manifold in the chosen system of
local coordinates.

Let us also recall an important geometrical structure on a Frobenius manifold $M$,
the {\sl intersection form} of $M$.
This is a symmetric
bilinear form $(~,~)^*$ on $T^*M$ defined by the formula
\beq (w_1,w_2)^*=i_E(w_1\cdot w_2),\eeq
here the product of two 1-forms $w_1$, $w_2$ at a point $t\in M$ is defined
by using the algebra structure on $T_tM$
and the isomorphism
\beq T_tM\to T_t^*M\eeq
 established by the invariant flat metric $<~,~>$. In the flat
 coordinates $t^1,\cdots,t^n$ of the invariant
 metric,  the intersection form can be represented by
 \beq
(dt^i,dt^j)^*=\mathcal{L}_EF^{ij}=(d-1+\hat{d}_i+\hat{d}_j)F^{ij},
 \eeq
where \beq F^{i j}=\eta^{i i'}\eta^{j j'}\dfrac{\p^2 F}{{\p
t^{i'}}{\p t^{j'}}}\eeq
and $F(t)$ is the potential of the Frobenius manifold.
Denote by $\Sigma_0\subset M$ the {\em discriminant} of $M$ on which the intersection form degenerates, then
an important property of the intersection form is that on $M\setminus\Sigma_0$ its inverse defines a new flat metric.


\prfMM From Theorem \ref{thm-du} and Theorem \ref{DZZ-thm4.4} we already know the existence of a flat metric $(\eta^{ij})$ and its flat local coordinates
$t^1, \dots, t^{l+1}$  on $\mathcal{M}_{k, m}(C_l)$.
By following the lines of the proof of Lemma 2.6 given in \cite{DZ1998} we
can show the existence of a unique weighted homogeneous
polynomial
$$G(t):=G(t^1,\dots,t^{k-1},t^{k+1},\dots,t^l,\frac1{t^{l-m}},\frac{1}{t^l},e^{t^{l+1}})$$
of degree $2$ such that the function
\beq\label{zz2}
F(t)=\frac{1}{2}(t^k)^2t^{l+1}+\frac{1}{2}t^k\dsum_{i,j\ne k}\eta_{i
j}\,t^i t^j+G(t)
\eeq satisfies the equations
\beq \label{zz3}
g^{i j}(t)=\mathcal{L}_E F^{i j}(t),\quad \Gamma^{i j}_m(t)=\td_j\,
c^{ij}_m(t),\quad i,j,m=1,\dots,l+1,
\eeq
where $c^{ij}_m(t)=\frac{\pal F^{ij}(t)}{\pal t^m}$. Obviously, the function $F$ satisfies the
equations
\beq
\frac{\pal^3 F(t)}{\pal t^k \pal t^i \pal t^j}=\eta_{ij},\quad i,j=1,\dots, l+1
\eeq and the
quasi-homogeneity condition \beq {\mathcal L}_E F=2 F+A_{\alpha\beta}t^\alpha t^\beta+B_\alpha t^\alpha+C, \eeq
where $A_{\alpha\beta}, B_{\alpha}, C$ are certain constants. From
the properties of a flat pencil of metrics \cite{Du1} it follows
that $F$ also satisfies the associativity equations
\beq\label{zz7}
c_{m}^{ij}(t)\,c_{q}^{m p}(t)=c_{m}^{ip}(t)\,c_{q}^{m j}(t)
\eeq
for any set of fixed indices $i,j,p,q$. From the definition of the Euler vector field we also have
\[\mathcal{L}_E e=-e.\]
 Thus we have constructed locally the Frobenius manifold structure on $\mathcal{M}_{k,m}(C_l)$.

\ian{It follows from the results of \cite{Du2} that the above locally defined Frobenius manifold structure is actually globally defined on $\mathcal{M}_{k,m}(C_l)$. In fact, by using the definition of the vector fields $e$, $E$ given respectively in \eqref{zh-18}, \eqref{zz1} and Corollary \ref{cor3.5}, we know that $(g_1^{ij})=(g^{ij}({y}))$, $(g_2^{ij})=(\eta^{ij}({y}))$ form a quasi-homogeneous flat pencil of metrics of degree $d=1$ in the sense of \cite{Du2} with $\tau=y^{l+1}$. By using Theorem 2.1 and the remark given before Example 2.1 of \cite{Du2}, we know that the multiplication rule of the Frobenius manifold structure that we defined above, in terms of the local flat coordinates $t^1, \dots, t^{l+1}$, can be represented in a coordinate free form by using the flat metric $\eta$, the intersection form $g$ and the Euler vector fields $E$.
From Proposition \ref{zh-9} and the formulae \eqref{zh-4}, \eqref{zh2-d} we know
that in the coordinates $\tilde{y}_1=y^1, \dots, \tilde{y}_{l}=y^l, \tilde{y}_{l+1}=e^{y^{l+1}}$ the components of the intersection form $g$ are polynomials of these coordinates, while the components
of the  $\eta_{\al\beta}(\tilde{y})$ of the flat metric $\eta$ are polynomials of
\[\tilde{y}_1, \dots, \tilde{y}_{l+1}, \frac{1}{\tilde{y}_{l+1}}, \frac{1}{\tau^l}, \frac{1}{\tau^{l-m}}\]
 when $m\ne 0, l-k$, and are polynomials of
\[\tilde{y}_1, \dots, \tilde{y}_{l+1}, \frac{1}{\tilde{y}_{l+1}}, \frac{1}{\tau^l}\]
when $m=0$ or $m=l-k.$
Thus it follows from Remark \ref{zh-11} that the Frobenius manifold structure is globally defined on $\mathcal{M}_{k, m}(C_l)$.
The theorem is proved. \epf}

 \begin{rem} By using Lemma \ref{add3.2} we know that we can also represent
the Frobenius manifold structure globally in the coordinates $\theta^0, \dots, \theta^l$ introduced in \ref{zh-19}, which correspond to the coordinates $a_0, \dots, a_l$ of
$\mathfrak{M}_{k,m,n}$ under the map $\mathfrak{h}$ given in Theorem \ref{Main2}.
\end{rem}}

\begin{rem}\label{rem5.4} It follows from Remark \ref{rem-1} that the
Frobenius manifold structures which correspond
to the integers $m$ and $l-k-m$ are equivalent. From the above
construction we see that the potential $F$ is in general a
polynomial of $t^1,\dots, t^{l+1}, \dfrac1{t^{l-m}}, \dfrac1{t^l},
e^{t^{l+1}}$, in the particular cases when $m=1$ and $m=l-k-1$ it
does not depend on $\frac1{t^{l}}$ and $\frac1{t^{l-1}}$
respectively. {When $k=l$, $m=0$, the Frobenius manifold structure coincides}
with the one that is constructed in \cite{DZ1998}.
\end{rem}

\subsection{Examples}

To end this section we give some examples to illustrate the above
construction of Frobenius manifold structures. For notational convenience, instead of $t^1,\dots,t^{l+1}$ we will
denote the flat coordinates of the metric $\eta^{ij}$ by
$t_1,\dots, t_{l+1}$, and we will also denote $\p_i=\frac{\p}{\p
t_i}$ in the the following examples.

\begin{ex}\label{ex6.1}
$[C_3,k=1]$ Let $R$ be the root system of type $C_3$, take $k=1$, then $d_1=d_2=d_3=1$, and
\begin{eqnarray*}
&&y^1={e^{2 i\pi x_{{4}}}} \left( \xi_{{1}}+\xi_{{2}}+\xi_{{3}} \right),\\
&&y^2={e^{2 i\pi x_{{4}}}} \left(
\xi_{{1}}\xi_{{2}}+\xi_{{1}}\xi_{{3}}+
\xi_{{2}}\xi_{{3}}
\right),\\
&&y^3={e^{2 i\pi x_{{4}}}}\xi_{{1}}\xi_{{2}}\xi_{{3}},\\
&&y^4=2 i\pi x_{{4}},
\end{eqnarray*}
where $\xi_j={e^{2 i\pi \left( x_{{j}}-x_{{j-1}} \right)
}}+{e^{-2 i\pi \left( x_{{j}}-x_{{j-1}} \right) }}$ and
$x_0=0$, $j=1, 2, 3$. The metric $(~,~)^{\sptilde}$ has the form
\begin{equation*}
((dx_i,dx_j)^{\sptilde})=\frac{1}{4\pi^2}\left( \begin
{array}{cccc} ~~1&~~1\,&~~1&~~0\\ ~~1& ~~2&~~2&~~0\\
~~1&~~2&~~3&~~0\\ ~~0&~~0&~~0&-1\end {array}
 \right).
 \end{equation*}

 \noindent {\bf Case I. $m=0$}, i.e., $e=\frac{\p}{\p y^1}-4\frac{\p}{\p y^2}+4\frac{\p}{\p y^3}$.

\noindent We first introduce the variables
 \eqa
&& z^1=y^{{1}}+6\,{e^{y^{{4}}}}, \ z^2=y^{{2}}+4\,y^{{1}}+12\,{e^{y^{{4}}}},\nn\\
&& z^3=y^{{3}}+2\,y^{{2}}+4\,y^{{1}}+8\,{e^{y^{{4}}}},\ z^4=y^4.\nn
 \eeqa
Then the flat coordinates are given by
$$t_1=z^{{1}}-2\,e^{z_4},\ t_2=(z^{{2}}-\frac{1}{6}\,z^{{3}})(z^{{3}})^{-\frac14},
\ t_3=(z^{{3}})^{\frac14},\ t_4=z^{{4}}$$
and the intersection form has the expression
\begin{eqnarray*}
&&g^{11}=2\,t_{{2}}t_{{3}}\,
{e^{t_{{4}}}}+\frac{1}{3}\,{t_{{3}}}^{4}{e^{t_{{4}}}}+4\,
 {e^{2t_{{4}}}} ,\\
 &&g^{12}=\frac{7}{3}\,{t_{{3}}}^{3}{e^{t_{{4}}}}+\frac{7}{2}\,t_{{2}}{e^{t_{{4}}}},~~
  g^{13}=\frac{5}{2}\,t_{{3}}{e^{t_{{4}}}}, ~~ g^{14}=t_1,\\
&& g^{22}=12\,{t_{{3}}}^{2}{e^{t_{{4}}}}-\frac{1}{4}\,{t_{{2}}}^{2}+\frac{1}{12}\,{t_{{3}}}^{3}t
_{{2}}-{\frac {1}{108}}\,{t_{{3}}}^{6}+\frac{1}{4}\,{\dfrac {{t_{{2}}}^{3}}{{t_
{{3}}}^{3}}},\nn\\
&&g^{23}=2\,t_{{1}}+4\,{e^{t_{{4}}}}-\frac{1}{3}\,t_{{2}}t_{{3}}+{\frac {1}{72}}\,{t_{{
3}}}^{4}-\frac{1}{4}\,{\dfrac {{t_{{2}}}^{2}}{{t_{{3}}}^{2}}},\\
&&g^{24}=\frac{3}{4}\,t_{{2}}, ~~g^{33}=\frac{1}{4}\,{\dfrac
{t_{{2}}}{t_{{3}}}}-\frac{1}{12}\,{t_{{3}}}^{2},~~
g^{34}=\frac{1}{4}\,t_{{3}}, ~~g^{44}=1.
\end{eqnarray*}
The potential has the form
\begin{eqnarray*}
&&F=\frac{1}{2}\,{t_{{1}}}^{2}t_{{4}}+\frac{1}{2}\,t_{{1}}t_{{2}}t_{{3}}
-{\frac {1}{48}}\,{t_{{ 2}}}^{2}{t_{{3}}}^{2}+{\frac
{1}{1440}}\, t_{{2}}{t_{{3}}}^{5}-{ \frac
{1}{36288}}\,{t_{{3}}}^{8}\\
&&\qquad
+t_{{2}}t_{{3}}{e^{t_{{4}}}}+\frac{1}{6}\,{t_{{3}}}^{4}{e^{t_{{4}}}}+\frac{1}{2}\,
e^{2t_4} +{\frac {1}{48}}\,{\dfrac {{t_{{2}}}^{3}}{t_{{3}}}
}
\end{eqnarray*}
and the Euler vector field is given by
$$E=t_1{\p_1}+\frac{3}{4}t_2{\p_2}+\frac{1}{4}t_3{\p_3}+{\p_4}.$$

\noindent {\bf Case II. $m=1$}, i.e., $e=\frac{\p}{\p y^1}-4\frac{\p}{\p y^3}$.

\noindent Define
\eqa
&&z^1=y^{{1}}+2\,{e^{y^{{4}}}}, \ z^2=\frac{1}{2}\,y^{{2}}+\frac{1}{4}\,y^{{3}}+y^{{1}}+2\,{e^{y^{{4}}}},\nn\\
&&z^3=\frac{1}{4}\,y^{{3}}-\frac{1}{2}\,y^{{2}}+y^{{1}}-2\,{e^{y^{{4}}}},\ z^4=y^4.\nn
\eeqa
Then the flat coordinates are
\eqa
t_1=z^1-2e^{z^4},\ t_2=\sqrt{z^2},\ t_3=\sqrt{z^3},\ t_4=z^4\nn
\eeqa
and the intersection form is given by
\eqa
&& g^{11}=2\,{t_{{2}}}^{2}{e^{t_{{4}}}}-2\,{t_{{3}}}^{2}{e^{t_{{4}}}}+4\,
 {e^{2t_{{4}}}},\nn\\
 &&g^{12}=3\,t_{{2}}{e^{t_{{4}}}},\ g^{13}=-3\,t_{{3}}{e^{t_{{4}}}},\ g^{14}=t_{{1}},  \nn\\
 &&g^{22}=2\,{e^{t_{{4}}}}+t_{{1}}-\frac{1}{4}\,{t_{{3}}}^{2}-\frac{1}{4}\,{t_{{2}}}^{2},
 \ g^{23}=-\frac{1}{2}\,t_{{2}}t_{{3}},\nn\\
 &&g^{33}=-2\,{e^{t_{{4}}}}+t_{{1}}-\frac{1}{4}\,{t_{{2}}}^{2}-\frac{1}{4}\,{t_{{3}}}^{2},\nn\\
&& g^{24}=\frac{1}{2}\,t_{{2}}, \ g^{34}=\frac{1}{2}\,t_{{3}},\ g^{44}=1.\nn
\eeqa
The potential has the expression
\eqa
&&F=\frac{1}{2}\,t_{{1}}{t_{{2}}}^{2}+\frac{1}{2}\,t_{{1}}{t_{{3}}}^{2}+\frac{1}{2}{t_{{1}}}^{2}\,t_{{4}}
-\frac{1}{48}\,{t_{{2}}}^{4}\nn \\
&&\qquad-\frac{1}{48}\,{t_{{3}}}^{4}-\frac{1}{8}\,{t_{{2}}}^{2}{t_{{3}}}^{2}+{t_{{2}}}^{2}{e^{t_{{4}}
}}-{t_{{3}}}^{2}{e^{t_{{4}}}}+\frac{1}{2}\,{e^{2t_{{4}}}}\nn
\eeqa
and the Euler vector field is given by
$$E=t_1{\p_1}+\frac{1}{2}t_2{\p_2}+\frac{1}{2}t_3{\p_3}+{\p_4}.$$
The Frobenius manifold structure that we obtain for this case is isomorphic
to the one given in Example 2.6 [$A_3,k=2$] of \cite{DZ1998}.
\end{ex}

\begin{ex}\label{ex6.2}
$[C_3,k=2]$ Let $R$ be the root system of type $C_3$, take $k=2$, then $d_1=1,d_2=d_3=2$, and
\begin{eqnarray*}
&&y^1={e^{2 i\pi x_{{4}}}} \left( \xi_{{1}}+\xi_{{2}}+\xi_{{3}} \right),\\
&&y^2={e^{2 i\pi x_{{4}}}} \left(
\xi_{{1}}\xi_{{2}}+\xi_{{1}}\xi_{{3}}+
\xi_{{2}}\xi_{{3}}
\right),\\
&&y^3={e^{2 i\pi x_{{4}}}}\xi_{{1}}\xi_{{2}}\xi_{{3}},\\
&&y^4=2 i\pi x_{{4}},
\end{eqnarray*}
where $\xi_j={e^{2 i\pi \left( x_{{j}}-x_{{j-1}} \right)
}}+{e^{-2 i\pi \left( x_{{j}}-x_{{j-1}} \right) }}$ and
$x_0=0$, $j=1,2,3$. The metric $(~,~)^{\sptilde}$ has the form
\begin{equation*}
((dx_i,dx_j)^{\sptilde})=\frac{1}{4\pi^2}\left( \begin
{array}{cccc} ~~1&~~1\,&~~1&~~0\\ ~~1& ~~2&~~2&~~0\\
~~1&~~2&~~3&~~0\\ ~~0&~~0&~~0&-\frac{1}{2}\end {array}
 \right).
 \end{equation*}

 \noindent {\bf Case I. $m=0$}, i.e., $e=\frac{\p}{\p y^2}-2\frac{\p}{\p y^3}$.
The Frobenius manifold structure that we obtain for this case  is isomorphic
to the one given in Example 2.7 [$B_3,k=2$] of \cite{DZ1998}.

\noindent {\bf Case II. $m=1$}, i.e., $e=\frac{\p}{\p y^2}+2\frac{\p}{\p y^3}$.

\noindent We first introduce the following variables
\eqa
&&z^1=y^{{1}}+2\,{e^{y^{{4}}}}, \ z^2=y^2+4e^{2y^4},\nn\\
&&z^3=2\,y^{{2}}-4\,y^{{1}}{e^{y_{{4}}}}-y^{{3}}+8\,{e^{2y^{{4}}}},\ z^4=y^4.\nn\eeqa
Then the flat coordinates given by
\eqa
t_1=z^1-4e^{z^4},\ t_2=z^{{2}}-2\,z^{{1}}{e^{z^{{4}}}}+6\,{e^{2\,z^{{4}}}},\ t_3=\sqrt{z^3},\ t_4=z^4.\nn
\eeqa
The potential has the expression
\eqa
&&F=\frac{1}{2}\,t_{{2}}{t_{{3}}}^{2}+\frac{1}{4}\,{t_{{1}}}^{2}t_{{2}}+\frac{1}{2}\,{t_{{2}}}^{2}
t_{{4}}-\frac{1}{48}\,{t_{{3}}}^{4}\nn\\
&&\qquad -{\frac {1}{96}}\,{t_{{1}}}^{4} +\, {t_{{3}}}^{2} {e^{2t_{{4}}}}-{t_{{3}}}^{2}t_{{1}}{e^{t_{{4}}}}
+\frac{1}{2}\,{t_{{1}}}^{2}{e^{2t_{{4}}}}+\frac{1}{4}\, {e^{4t_{{4}}}}\nn
\eeqa
and the Euler vector field is given by
$$E=\frac{1}{2}t_1{\p_1}+t_2{\p_2}+\frac{1}{2}t_3{\p_3}+\frac{1}{2}{\p_4}.$$
This Frobenius manifold structure is exactly the one given in Example 2.7 [$B_3,k=2$] of \cite{DZ1998}.
\end{ex}

\begin{ex} \label{ex6.3} $[C_4,k=1,m=0]$
Let $R$ be the root system of type $C_4$, take $k=1$, then
$d_1=d_2=d_3=d_4=1$, and
\begin{eqnarray*}
&&y^1={e^{2\,i\pi\,x_{{5}}}} \left( \xi_{{1}}+\xi_{{2}}+\xi_{{3}}+\xi_{{4}} \right),\\
&&y^2={e^{2\,i\pi\,x_{{5}}}} \dsum_{1\leq a<b\leq 4}
\xi_{{a}}\xi_{{b}},\\
&&y^3={e^{2\,i\pi\,x_{{5}}}} \dsum_{1\leq a<b<c\leq 4}
\xi_{{a}}\xi_{{b}}\xi_{{c}},\\
&&y^4={e^{2\,i\pi\,x_{{5}}}}\xi_{{1}}\xi_{{2}}\xi_{{3}}\xi_{{4}}, \\
&&y^5=2\,i\pi\,x_{{5}},
\end{eqnarray*}
where $\xi_j={e^{2 i\pi \left( x_{{j}}-x_{{j-1}} \right)
}}+{e^{-2 i\pi \left( x_{{j}}-x_{{j-1}} \right) }}$ and
$x_0=0$, $j=1,2,3,4$. The metric $(~,~)^{\sptilde}$ has the form
\begin{equation*}
((dx_i,dx_j)^{\sptilde})=\frac{1}{4\pi^2}\left( \begin
{array}{rrrrr}
~~1&~~1&~~1&~~1&~~0\\\noalign{\medskip}1&2&2&2&0\\\noalign{\medskip}1&2&3&3&0\\
\noalign{\medskip}1&2&3&4&0\\ \noalign{\medskip}0&0&0&0&-1\end
{array}
 \right).\end{equation*}
Introduce the variables
\eqa
&& z^1=y^{{1}}+8\,{e^{y^{{5}}}},\ z^2=y^{{2}}+6\,y^{{1}}+24\,{e^{y^{{5}}}},\nn\\
&&z^3=y^{{3}}+4\,y^{{2}}+12\,y^{{1}}+32\,{e^{y^{{5}}}},\ z^5=y^5,\nn\\
&&z^4=y^{{4}}+2\,y^{{3}}+8\,y^{{1}}+4\,y^{{2}}+16\,{e^{y^{{5}}}},
\nn \eeqa
and
\eqa
&&w_1=z^1-2e^{z^5},\
w_2=(z^2-\frac{1}{6}\,z^3+\frac{1}{30}\,z^4)(z^4)^{-\frac16},\nn\\
&& w_3=(z^3-\frac{1}{4}\,z^4) (z^4)^{-\frac23},\
w_4=(z^4)^{\frac16},\ w_5=z^5.\nn
\eeqa
Then we have the expression of the flat coordinates
\beq
t_1=w_1,\
t_2=w_2-\frac{1}{12}w_3^2\,w_4,\ t_3=w_3 w_4,\ t_4=w_4,\ t_5=w_5.\nn
\eeq
The potential $F$ is given by
\begin{eqnarray*}
&&F=\frac
{1}{2}\,{t_{{1}}}^{2}t_{{5}}+{\frac {1}{2}}\,t_{{1}}
t_{{2}}t_{{4}}-{\frac {1}{6912}}
\,{t_{{3}}}^{4}+{\frac{1}{17280}}\,{t_{{3}}}^{3}{t_{{4}}}^{3}\\
&&\qquad -{\frac
{1}{288}}\,t_{{2}}t_{{4}}{t_{{3}}}^{2}-{\frac
{1}{34560}}\,{t_{{4}}}^{ 6}{t_{{3}}}^{2}+{\frac
{1}{24}}\,t_{{1}}{t_{{3}}}^{2}+{\frac
{1}{1440}}\,t_{{3}}{t_{{4}}}^{4 }t_{{2}}\\
&&\qquad -{ \frac
{1}{48}}\,{t_{{2}}}^{2}{t_{{4}}}^{2}-{\frac
{1}{60480}}\,{t_{{4}}}^{7}t_{{2}}+{\frac {1}{
345600}}\,{t_{{4}}}^{9}t_{{3}}-{\frac {1}{
7603200}}\,{t_{{4}}}^{12}\\
&&\qquad+{\frac {1}{12}
}\,{e^{t_{{5}}}}{t_{{3}}}^{2}+\frac
{1}{6}\,{e^{t_{{5}}}}t_{{3}}{t_{{4}}}^{3}+{\frac
{1}{120}}\,{e^{t_{{5}}} }{t_{{4}}}^{6}+
 t_{{2}}t_{{4}}{e^{t_{{5}}}}+\frac {1}{2}\,
{e^{2t_{{5}}}}\\
&&\qquad +{\frac {1}{24}}\,{\dfrac
{t_{{3}}{t_{{2}}}^{2}}{t_{{4}}}}-{\frac {1} {216}}\,{\dfrac
{t_{{2}}{t_{{3}}}^{3}}{{t_{{4}}}^{2}}}+{ \frac
{1}{4320}}\,{\dfrac {{t_{{3}}}^{5}}{{t_{{4}}}^{3}}}
\end{eqnarray*}
with the Euler vector field
$$E=t_1{\p_1}+\frac{5}{6}t_2{\p_2}+\frac{1}{2}t_3{\p_3}+\frac{1}{6}t_4{\p_4}+{\p_5}.$$
\end{ex}

\begin{ex}\label{ex6.4} $[C_4,k=2,m=0] \ $Let $R$ be the root system of type $C_4$, take $k=2$, then
$d_1=1,d_2=d_3=d_4=2$, and
\begin{eqnarray*}
&&y^1={e^{2 i\pi x_{{5}}}} \left( \xi_{{1}}+\xi_{{2}}+\xi_{{3}}+\xi_{{4}} \right),\\
&&y^2={e^{4 i\pi x_{{5}}}} \dsum_{1\leq a<b\leq 4}
\xi_{{a}}\xi_{{b}},\\
&&y^3={e^{4 i\pi x_{{5}}}} \dsum_{1\leq a<b<c\leq 4}
\xi_{{a}}\xi_{{b}}\xi_{{c}},\\
&&y^4={e^{4 i\pi x_{{5}}}}\xi_{{1}}\xi_{{2}}\xi_{{3}}\xi_{{4}}, \\
&&y^5=2 i\pi x_{{5}},
\end{eqnarray*}
where $\xi_j$ are defined as in the last example. The metric $(~,~)^{\sptilde}$ has the form
\begin{equation*}
((dx_i,dx_j)^{\sptilde})=\frac{1}{4\pi^2}\left( \begin
{array}{rrrrr}
~~1&~~1&~~1&~~1&~~0\\\noalign{\medskip}1&2&2&2&0\\\noalign{\medskip}1&2&3&3&0\\
\noalign{\medskip}1&2&3&4&0\\
\noalign{\medskip}0&0&0&0&-\frac{1}{2}\end {array}
 \right).\end{equation*}
Introduce the following variables
\eqa
&&z^1=y^{{1}}+8\,{e^{y^{{5}}}}, \ z^5=y^5,\nn\\
&&z^2=y^{{2}}+6\,y^{{1}}{e^{y^{{5}}}}+24\, {e^{2y^{{5}}}},\nn\\
&&z^3=y^{{3}}+4\,y^{{2}}+12\,y^{{1}}{e^{y^{{5}}}}+32\, {e^{2y^{{5}}}},\nn\\
&&z^4=y^{{4}}+2\,y^{{3}}+4\,y^{{2}}+8\,y^{{1}}{e^{y^{{5}}}}+16\, {e^{2y_{{5}}}}.\nn
\eeqa
Then the flat coordinates are given by
\eqa
&&t_1=z^{{1}}-4e^{z^5},\ t_2=z^2-2z^1e^{z^5}+6e^{2z^5},\nn\\
&& t_3=(z^3-\frac{1}{6}\,z^4)(z^4)^{-\frac14},\ t_4=(z^4)^{\frac14},\ t_5=z^5.\nn
\eeqa
The Euler vector field and the potential are given respectively by
\eqa
&&E=\frac{1}{2}t_1{\p_1}+t_2{\p_2}+\frac{3}{4}t_3{\p_3}+\frac{1}{4}t_4{\p_4}+\frac{1}{2}{\p_5}.\nn\\
&&F=\frac{1}{2}\,{t_{{2}}}^{2}t_{{5}}+\frac{1}{4}\,{t_{{1}}}^{2}t_{{2}}+\frac{1}{2}\,t_{{4}}t_{{3
}}t_{{2}}+{\frac {1}{1440}}\,{t_{{4}}}^{5} t_{{3}}-{\frac
{1}{48}}\,{t_{{4}}}^{2}{t_{{3}}}^{2}\nn\\
&&\qquad-{\frac
{1}{36288}}\,{t_{{4}}}^{8}-{\frac {1}{96}}\,{t_{{1}}}^{4}
+\frac{1}{2}\,{e^{2\,t_{{5}}}}{t_{{1}}}^{2}+\frac{1}{6}\,{e^{t_{{5}}}}t_{{1}}{t_{{4}}}^{
4}+\frac{2}{3}\,{t_{{4}}}^{4}{e^{2\,t_{{5}}}}\nn\\
&&\qquad+{e^{t_{{5}}}}t_{{1}}t_{{3}}t_{{4}}+t_{{3}}t_{{4}}{e^{2\,t
_{{5}}}}+\frac{1}{4}\,{e^{4\,t_{{5}}}}+{\frac {1}{48
}}\,{\dfrac {{t_{{3}}}^{3}}{t_{{4}}}}.\nn
\eeqa
\end{ex}

In the following, we present more examples and omit all computations and only list the
potentials and the Euler vector fields.

\begin{ex}$[C_5,k=1,m=2]$
Let $R$ be the root system of type $C_5$, take $k=1, m=2$, then
\eqa
&&F=\frac{1}{2}\,t_{{6}}{t_{{1}}}^{2}+\frac{1}{2}\,t_{{1}}t_{{2}}t_{{3}}+\frac{1}{2}\,t_{{1}}t_{{4}}t_{{5}}
-{\frac {1}{72}}\,{t_{{3}}}^{4}{t_{{5}}}^{4}
-\frac{1}{8}\,t_{{2}}t_{{3}}t_{{4}}t_{{5}}\nn\\
&&\qquad -{\frac
{1}{2268}}\,{t_{{5}}}^{8}-{\frac {1}{36288}}\,{t_{{3}}}^{8}
-\frac{1}{48}\,{t_{{3}}}^{2}{t_{{2}}}^{2}-\frac{1}{48}\,{t_{{4}}}^{2}{t_{{5}}}^{2}
+\frac{1}{24}\,{t_{{5}}}^{4}t_{{2}}t_{{3}}\nn\\
&&\qquad +{\frac {1}{96}}\,{t_{{3}}}^{4}t_{{4}}t_{{5}}+{\frac
{1}{1440}}\,{t _{{3}}}^{5}t_{{2}}+{\frac
{1}{360}}\,t_{{4}}{t_{{5}}}^{5}+t_{{2}}t_{{3}}{e^{t_{{6}}}}-
t_{{4}}t_{{5}}{e^{t_{{6}}}}\nn\\
&&\qquad -\frac{2}{3}\,{t_{{5}}}^{4}{e^{t_{{6}}}}+\frac{1}{6}\,{t_{{3}}}^{4}{e^{t_{{6}}}}+\frac{1}{2}\,
{e^{2t_{{6}}}}+\frac{1}{48}\,{\frac {{t_{{2}}}^{3}}{t_{{3}}}}+{\frac
{1}{192}}\,{\frac {{t_{{4}}}^{3}}{t_{{5}}} }.\nn
\eeqa
The Euler vector field is given by
$$E=t_1{\p_1}+\frac{3}{4}t_2{\p_2}+\frac{1}{4}t_3{\p_3}+\frac{3}{4}t_4{\p_4}+\frac{1}{4}t_5{\p_5}+{\p_6}.$$
\end{ex}

\begin{ex}$[C_6,k=1,m=2]$
Let $R$ be the root system of type $C_6$, take $k=1$, then

\eqa
&&F= \frac{1}{2}\,{t_{{1}}}^{2}t_{{7}}+\frac{1}{24}\,
t_{{1}}{t_{{3}}}^{2}+\frac{1}{2}\,t_{{1}}t_{{2}}t_{{4
}}+\frac{1}{2}\,t_{{1}}t_{{5}}t_{{6}}-\frac{1}{48}\,{t_{{2}}}^{2}{t_{{4}}}^{2}\nn \\
&&\qquad+{\frac {1}{17280}}\,{t_{{4}}}^{
3}{t_{{3}}}^{3}-\frac{1}{48}\,{t_{{5}}}^{2}{t_{{6}}}^{2}+{\frac
{1}{360}}\,t_{{5}}{t_{{6}}}^{5}
+{\frac {1}{288}}\,{t_{{3}}}^{2}{t_{{6}}}^{4}\nn \\
&&\qquad+{\frac {17}{5760}}\,{t_{{6}}}^{4}{t_{{4}}}^{6}-{\frac
{1}{60480}}\,{t_{{4}}}^
{7}t_{{2}}-{\frac{1}{72}}\,{t_{{6}}}^{4}{t_{{4}}}^{3}t_{{3}}-{\frac {1}{288}}\,t_{{2}}{t_{{3}}}^{2}t_{{4}}\nn \\
&&\qquad+{\frac {1}{1440}}\,t_{{2}}t_{{3}}{t_{{4}}}^{4}-{\frac
{1}{96}}\,{t_{{3}}}^{2}t_{{5}}t_{{6} }-{\frac
{1}{2268}}\,{t_{{6}}}^{8}-{\frac {1}{34560}}\,{t_{{4}}}^{6}{t
_{{3}}}^{2}\nn \\
&&\qquad-{\frac {1}{6912}}\,{t_{{3}}}^{4}-{\frac {1}{7603200}}\,{t_
{{4}}}^{12}+\frac{1}{24}\,{t_{{6}}}^{4}t_{{2}}t_{{4}}-{\frac {1}{960}}\,t_{{6}}{t_{{4}}}^{6}t_{{5}}\nn\\
&&\qquad+{\frac
{1}{345600}}\,{t_{{4}}}^{9}t_{{3}}-\frac{1}{8}\,t_{{6}}t_{{2}}t_{{4}}t_{{5}}
+{\frac {1}{96}}\,t_{{6}}{t_{{4}}}^{3}t_{{3}}t_{{5
}}+\frac{1}{6}\,{t_{{4}}}^{3}t_{{3}}{e^{t_{{7}}}}\nn\\
&&\qquad+{\frac {1}{120}}\,{t_{{4}}}^{6}{e^{t_{{7}}}}+t_{{2
}}t_{{4}}{e^{t_{{7}}}}-t_{{5}}t_{{6}}{e^{t_{{7}}}}+\frac{1}{12}\,{t_{{3}
}}^{2}{e^{t_{{7}}}}-\frac{2}{3}\,{t_{{6}}}^{4}{e^{t_{{7}}}}\nn\\
&&\qquad+\frac{1}{2}\,{e^{2\,t_{{7}}}}+\frac{1}{24}\,{\frac {{t_{{2}}}^{2
}t_{{3}}}{t_{{4}}}}-{\frac {1}{216}}\,{\frac
{t_{{2}}{t_{{3}}}^{3}}{{t_{{4}}}^{2}}} +{\frac {1}{4320}}\,{\frac
{{t_{{3}}}^{5}}{{t _{{4}}}^{3}}}+{\frac {1}{ 192}}\,{\frac
{{t_{{5}}}^{3}}{t_{{6}}}},\nn
\eeqa
and the Euler vector field is given by
$$E=t_1{\p_1}+\frac{5}{6}t_2{\p_2}+\frac{1}{2}t_3{\p_3}+\frac{1}{6}t_4{\p_4}+\frac{3}{4}t_5{\p_5}
+\frac{1}{4}t_6{\p_6}+{\p_7}.$$
\end{ex}


\section{On the Frobenius manifold structures related to the root system of type $B_l$ and $D_l$}
\label{sec-6}

For the root system $R$ of type $B_l$, we also define an indefinite
metric $(~,~)^{\sptilde}$ on $\widetilde{V}_\mathbb{C} = \widetilde{V} \otimes_\mathbb{R} \mathbb{C}$
where $\widetilde V$ is the orthogonal direct sum of $V$ and $\mathbb
R$. Here $V$ is endowed with the $W$-invariant Euclidean metric
\beq
(d x_s, dx_n)^{\sptilde}=\frac{1}{4\pi^2}[(1-\frac12\delta_{n,l}) s-\frac{l}4\,\delta_{n,l}\delta_{s,l}],
\quad 1\le s\le n\le l
\eeq
and $\mathbb R$ is endowed with the metric
\beq
(dx_{l+1},dx_{l+1})^{\sptilde}=-\frac1{4\pi^2 d_k}.
\eeq
Here the numbers $d_k$ are defined in \eqref{addDZ2.1} and \eqref{addDZ2.2}.
The basis of the $W_a$-invariant Fourier polynomials $y_1(\bx),\dots,y_{l-1}(\bx)$,
 $y_l(\bx)$ are defined in \eqref{add1.10-a}--\eqref{add1.10-c}.
The generators of the ring $\widetilde{W}^{(k)}(B_l)$ have the same form as that of \eqref{ip-a} and \eqref{ip-b}.
It is easy to see that the components of the resulting metric $(g^{ij}(y))$ coincide
with those corresponding
to the root system of type $C_l$ if we perform the change of coordinates \ian{
\eqa
&& y^j\mapsto \bar y^j=y^j,\ y^{l+1}\mapsto \bar{y}^{l+1}=y^{l+1},
~~ j=1,\dots,l-1,\nn\\
 && y^l\mapsto {\bar y}^l=(y^l)^2-\dsum_{j=0}^{l-1}2^{l-s}y^s e^{(k-d_s)y^{l+1}}   \eeqa
for $1\le k\le l-1$ and
\eqa
&& y^j\mapsto \bar y^j=y^j,\ y^{l+1}\mapsto \bar{y}^{l+1}=\frac12 y^{l+1},
~~j=1,\dots,l-1,\nn\\
 && y^l\mapsto {\bar y}^l=(y^l)^2-\dsum_{j=0}^{l-1}2^{l-s}y^s e^{\frac{1}{2}(l-s)y^{l+1}}    \eeqa}
for the case when $k=l$. Thus, the Frobenius manifold structure that we obtain in this way from $B_l$,
by fixing the $k$-th vertex of the corresponding Dynkin diagram,
is isomorphic to the one that we obtain from $C_l$ by choosing the $k$-th vertex of the Dynkin diagram of $C_l$.


For the root system $R$ of type $D_l$, the indefinite  metric $(~,~)^{\sptilde}$
on $\widetilde V=V\oplus\mathbb R$ is defined through  the $W$-invariant Euclidean metric
\eqa\label{add8.2}
&&(d x_s,
dx_n)_1^{\sptilde}=\dfrac{s}{4\pi^2},\quad 1\le s\le n\le l-2,\nn\\
&&(d x_s, dx_n)^{\sptilde}=\frac{s}{8\pi^2},\quad 1\le s \le l-2, n=l-1, l-2, \\
&&(d x_{l-1}, dx_{l-1})^{\sptilde}=(d x_{l}, dx_{l})^{\sptilde}=\dfrac{l}{16\pi^2},\quad
(d x_{l-1}, dx_{l})^{\sptilde}=\dfrac{l-2}{16\pi^2},\nn
\eeqa
and  \beq \label{add8.3}
(dx_{l+1},dx_{l+1})^{\sptilde}=-\frac1{4\pi^2 d_k}.
 \eeq Here the numbers $d_k$ are defined in \eqref{addDZ2.9}.
  The set of generators for the ring ${\mathcal A}={\mathcal A}^{(k)}(D_l)$ have
  the same form as that of \eqref{ip-a} and \eqref{ip-b},
where $y_j(\bx)$ are  defined  in \eqref{add1.15-a} and \eqref{add1.15-b}.
It can be verified that the components of the resulting metric $(g^{ij}(y))$ coincide
with those corresponding to the root system of type $C_l$ if we perform the change of coordinates \ian{
\eqa
&&y^j \mapsto \bar y^j=y^j, ~~j=1,\cdots, l-2,~l+1,\nn \\
&&y^{l-1}\mapsto \bar
y^{l-1}=y^{l-1}y^{l}-\frac{1}{4}\dsum_{s=0}^{l-2}[2^{l-s}-(-2)^{l-s} ]\ y^{s}e^{(k-d_s)y^{l+1}},\\
&&y^{l}\mapsto \bar
y^{l}\mapsto \bar
y^{l}=(y^l)^2+(y^{l-1})^2-\frac{1}{2}\dsum_{s=0}^{l-2}[2^{l-s}+(-2)^{l-s} ]\ y^{s}e^{(k-d_s)y^{l+1}}
 \nn\eeqa
 for $1\le k\le l-2$ and}
 \ian{\eqa
&&y^j \mapsto \bar y^j=y^j, ~~j=1,\cdots, l-2; \quad y^{l+1} \mapsto \bar y^{l+1}=2y^{l+1}, \nn\\
&&y^{l-1}\mapsto \bar
y^{l-1}=y^{l-1}y^{l}-\frac{1}{4}\dsum_{s=0}^{l-2}[2^{l-s}-(-2)^{l-s} ]\ y^{s}e^{\frac{1}{2}(l-s-1)y^{l+1}},\\
&&y^{l}\mapsto \bar
y^{l}=(y^l)^2e^{\frac{1}{2}y^{l+1}}+(y^{l-1})^2e^{-\frac{1}{2}y^{l+1}}-\frac{1}{2}\dsum_{s=0}^{l-2}[2^{l-s}+(-2)^{l-s} ]
\ y^{s}e^{\frac{1}{2}(l-s-1)y^{l+1}}
 \nn\eeqa
 for the case $k=l-1$ and }
 \ian{\eqa
&&y^j \mapsto \bar y^j=y^j, ~~j=1,\cdots, l-2; \quad y^{l+1} \mapsto \bar y^{l+1}=2y^{l+1}, \nn\\
&&y^{l-1}\mapsto \bar
y^{l-1}=y^{l-1}y^{l}-\frac{1}{4}\dsum_{s=0}^{l-2}[2^{l-s}-(-2)^{l-s} ]\ y^{s}e^{\frac{1}{2}(l-s-1)y^{l+1}},\\
&&y^{l}\mapsto \bar
y^{l}=(y^l)^2+(y^{l-1})^2e^{y^{l+1}}-\frac{1}{2}\dsum_{s=0}^{l-2}[2^{l-s}+(-2)^{l-s} ]\ y^{s}e^{\frac{1}{2}(l-s)y^{l+1}}
 \nn\eeqa
 for the case $k=l$.} Thus, the Frobenius manifold structure
that we obtain in this way from $D_l$, by fixing the $k$-th vertex
of the corresponding Dynkin diagram, is isomorphic to the one that
we obtain from $C_l$ by choosing the $k$-th vertex of the Dynkin
diagram of $C_l$.

\section{LG superpotentials for the Frobenius manifolds of $\mathcal{M}_{k,m}(C_l)$-type}\label{sec-7}

We consider a particular class of \ian{LG superpotentials consisting of} cosine-Laurent series of one variable with tri-degree $(2k,2m,2n)$,
these being functions of the form\footnote{When $k=1$ and $m=n=0$, this reduces to
$\lambda(\varphi)=a_1+a_0\cos^2(\varphi)$.
If we set
$$\cos^2(\varphi)=\frac{1+\cos(2\varphi)}{2}, \quad a_0=-4e^{\frac{t_2}{2}}, \quad
a_1=t_1+2e^{\frac{t_2}{2}},\quad p=2\varphi,$$
then the LG superpotential is rewritten as
$$\lambda(p)=t_1-2e^{\frac{t_2}{2}}\cos(p),$$
which is exactly the LG superpotential of the $\mathbb{CP}^1$-model obtained in Example I.1
\cite{Du1}.}
{\eqa
\lambda(\varphi)=\left(\cos^{2}(\varphi)-1\right)^{-m}
\dsum_{j=0}^{k+m+n} a_j \cos^{2(k+m-j)}(\varphi),
 \label{fm2.1}
\eeqa}
where all $a_j \in \mathbb{C}$, $m,n\in \mathbb{Z}_{\geq 0}$ and $k\in \mathbb{N}$.
The cosine is considered as an analytic function on the cylinder $\varphi\simeq \varphi+2\pi$.
\ian{
We denote by $\mathfrak{M}_{k,m,n}$ the space of this kind of cosine Laurent series
with the following conditions:
\eqa
&&a_0 a_{k+m+n}\ne 0,\quad \textrm{when}\ m= 0;\label{cnd-a-1}\\
&&a_0 \ne 0, \ \sum_{j=0}^{k+m+n}a_j \ne 0,\quad \textrm{when}\
n= 0;\label{cnd-a-2}\\
&&a_0 a_{k+m+n}\ne 0, \ \sum_{j=0}^{k+m+n}a_j\ne 0,\quad \textrm{when}\ m n\ne 0.\label{cnd-a-3}
\eeqa
}
By analogy  with the construction in \cite{Du1,Ber3,Z2007},
the space $\mathfrak{M}_{k,m,n}$ carries a natural structure of Frobenius manifold.
The invariant inner product
$\eta$ and the intersection form $g$ of two vectors $\p'$, $\p''$
tangent to $\mathfrak{M}_{k,m,n}$ at a point $\lambda(\varphi)$ can be defined by
the following formulae
\eqa
&&{\eta}(\p',\p'')=(-1)^{k+1}\dsum_{|\lambda|<\infty}
\res_{d\lambda=0}
\dfrac{\p'(\lambda(\varphi)d\varphi)\p''(\lambda(\varphi)d\varphi)}{d\lambda(\varphi)},
\label{fm2.3}\\
&&{g}(\p',\p'')=-\dsum_{|\lambda|<\infty}\res_{d\lambda=0}
\dfrac{\p'(\log\lambda(\varphi)d\varphi)\p''(\log\lambda(\varphi)d\varphi)}{d\log\lambda(\varphi)}.
\label{fm2.4}\eeqa
In these formulae, the derivatives
$\p'(\lambda(\varphi)d\varphi)$ $etc.$ are to be calculated keeping $\varphi$ fixed.
The formulae \eqref{fm2.3} and \eqref{fm2.4} uniquely determine
multiplication of tangent vectors on $\mathfrak{M}_{k,m,n}$ assuming that the
Euler vector field $E$  has the form
\eqa
E=\dsum_{j=0}^{k+m+n}a_j\frac{\p}{\p a_j}.\label{fm2.5}\eeqa
For tangent vectors $\p'$, $\p''$ and $\p'''$ to $\mathfrak{M}_{k,m,n}$  , one has \beq
c(\p',\p'',\p''')=-\dsum_{|\lambda|<\infty} \res_{d\lambda=0}
\dfrac{\p'(\lambda(\varphi)d\varphi)\p''(\lambda(\varphi)d\varphi)\p'''
(\lambda(\varphi)d\varphi)}{d\lambda(\varphi)d\varphi}.
\label{fm2.6}\eeq
The canonical coordinates $u_1,\cdots, u_{k+m+n+1}$ for this multiplication are
the critical values of $\lambda(\varphi)$ and
\beq \p_{u_\alpha}\cdot \p_{u_\beta}=\delta_{\alpha\beta}\p_{u_\alpha}, \quad
\mbox{where}\quad \p_{u_\alpha}=\frac{\p}{\p {u_\alpha}}\label{fm2.7}\eeq
\ian{(these are defined only on the semi-simple locus of the manifold).}

For clarity, we use the notation
\eqa && \lambda(P)=(P^2-1)^{-m} \dsum_{j=0}^{l} a_j P^{2(k+m-j)}, \quad l=k+m+n,
\nn\\
&& \dot{\lambda}(P)=\frac{d\lambda(P)}{dP},\quad P=\cos(\varphi),\quad P'(\varphi)=\dfrac{dP}{d\varphi}=-\sin(\varphi)
\label{fm2.8} \eeqa
and
\eqa \lambda(\varphi)=a_0 (P^2-1)^{-m}  P^{-2n} \prod_{j=1}^l(P^2-p_j^2), \quad a_0=e^{2k\bi \varphi_{l+1}},\quad
p_j=P(\varphi_j).\label{fm2.9} \eeqa
Here $P(\varphi)$ has no relation with the function $P(u)$ used in \eqref{P1}. Without confusion, we always use $\lambda(P)$ instead of $\lambda(\varphi)$.
\ian{On comparing coefficients in the two expansions (\ref{fm2.1}) and (\ref{fm2.9}) of the superpotential we obtain expressions for the $a_i$ in terms of $a_0$ and the $p_j\,.$}

Before proceeding to the main result, we give some useful identities.

\begin{lem}
\beq \lambda'(\varphi_j)=\left.\dfrac{2P P'(\varphi) \lambda(\varphi)}{P^2-p_j^2}\right|_{\varphi=\varphi_j},\quad
j=1,\cdots,l.\label{fm2.16} \eeq
\end{lem}
\begin{proof} This follows from
\beq \lambda'(\varphi)=2PP'(\varphi)\lambda(\varphi)\left(\dsum_{j=1}^l\frac{1}{P^2-p_j^2}
-\frac{m}{P^2-1}-\frac{n}{P^2}\right) \nn\eeq
and the definition of $\lambda(\varphi)$ in \eqref{fm2.9}.
 \end{proof}

Let us factorize
\beq
\lambda'(\varphi)=2k a_0 (P^2-1)^{-m-1} P^{-2n-1}P'(\varphi) \prod_{\alpha=1}^{l+1}(P^2-q_\alpha^2), \quad
q_\alpha=P(\psi_\alpha),\label{fm2.13}\eeq
where all $q_\alpha^2$ are distinct. When $m=0$, we choose $P'(\psi_{l+1})=0$, that is to say,
 $$\psi_{l+1}=0,\pi, \quad \mbox{i.e.,}\quad q_{l+1}=P(\psi_{l+1})=1.$$

\begin{lem} \label{lem1.2} For $1 \leq \alpha \leq l+1$, we have
\beq
\lambda''(\psi_\alpha)=\left.\dfrac{c_{\alpha,m} P P'(\varphi) \lambda'(\varphi)}{P^2-q_\alpha^2}\right|_{\varphi=\psi_\alpha},\quad
c_{\alpha,m}=2-\delta_{\alpha,l+1}\delta_{m,0}. \label{fm2.17}
\eeq
\end{lem}

\begin{proof} By definition, we have
\eqa\lambda''(\varphi)&=& 2k a_0 \dfrac{d}{d\varphi}\left((P^2-1)^{-m-1} P^{-2n-1}\right) \prod_{\alpha=1}^{l+1}(P^2-q_\alpha^2) P'(\varphi)\nn\\
&+& 2k a_0\, (P^2-1)^{-m-1} P^{-2n-1}\dfrac{d}{d\varphi}\left(\prod_{\alpha=1}^{l+1}(P^2-q_\alpha^2) \right)P'(\varphi)\nn\\
&+& 2k a_0 (P^2-1)^{-m-1} P^{-2n-1}\prod_{\alpha=1}^{l+1}(P^2-q_\alpha^2)\,\dfrac{d^2P}{d^2\varphi}\nn\\
&=& \dsum_{\alpha=1}^{l+1} \dfrac{2PP'(\varphi)\lambda'(\varphi)}{P^2-q_\alpha^2}-\frac{(2n+1)P'(\varphi)\lambda'(\varphi)}{P}
+\dfrac{(2m+1) P\lambda'(\varphi)}{P'(\varphi)}.\nn\eeqa
So, with the use of \eqref{fm2.13}, we get
\eqa \lambda''(\psi_\alpha)&=&\left.\left(\dsum_{\alpha=1}^{l+1} \dfrac{2PP'(\varphi)\lambda'(\varphi)}{P^2-q_\alpha^2}
+\dfrac{(2m+1) P\lambda'(\varphi)}{P'(\varphi)}\right)\right|_{\varphi=\psi_\alpha}\nn\\
&=& \left\{\begin{array}{ll}
\left.\dfrac{2P P'(\varphi) \lambda'(\varphi)}{P^2-q_\alpha^2}\right|_{\varphi=\psi_\alpha}, &\quad \alpha=1,\cdots, l,\\
-\left.\dfrac{P \,\lambda'(\varphi)}{P'(\varphi)}\right|_{\varphi=\psi_{l+1}},& \quad \alpha=l+1,\quad m=0,\\
\left.\dfrac{2P P'(\varphi) \lambda'(\varphi)}{P^2-q_\alpha^2}\right|_{\varphi=\psi_{l+1}}, &\quad \alpha=l+1,\quad m\ne 0\\
\end{array}\right. \nn\\
&=& \left.\dfrac{c_{\alpha,m} P P'(\varphi) \lambda'(\varphi)}{P^2-q_\alpha^2}\right|_{\varphi=\psi_\alpha}.
\nn\eeqa
Thus the lemma is proved.\end{proof}

We define canonical coordinates
$$u_\alpha=\lambda(\psi_\alpha),\quad \alpha=1,\cdots, l+1,$$
then
\beq
  \p_{u_\alpha}\lambda(\varphi)|_{\varphi=\psi_\beta}=\delta_{\alpha\beta}.
\label{fm2.14}\eeq
Observe that
$$(P^2-1)^{m}P^{2n}\p_{u_\alpha}\lambda(P)=(\p_{u_\alpha}a_0) P^{2l}+\cdots +\p_{u_\alpha}a_l$$
is a polynomial of $P$ and
$$(P^2-1)^{m} P^{2n}\p_{u_\alpha}\lambda(P)|_{P=q_\beta}=(q_\beta^2-1)^{m+1} q_\beta^{2n-1}\delta_{\alpha\beta},$$
we thus obtain, using the Lagrange interpolation formula,
\beq
 \p_{u_\alpha}\lambda(\varphi)=\frac{c_{\alpha,m} P P'(\varphi)}{P^2-q_\alpha^2}\frac{\lambda'(\varphi)}{\lambda''(\psi_\alpha)}, \quad \alpha=1,\cdots,l+1.
 \label{fm2.15}\eeq

\begin{lem}
 \beq \p_{u_\alpha}\varphi_\beta
 = \left\{\begin{array}{l}
- \dfrac{c_{\alpha,m} p_\beta  P'(\varphi_\beta)}{\lambda''(\psi_\alpha)\,(p_\beta^2-q_\alpha^2)}, \beta=1,\cdots, l,\\
\dfrac{1}{2k\bi}\left(\dfrac{\delta_{\alpha,l+1}}{\lambda(\psi_{\alpha})}+\dfrac{2 c_{\alpha,m} }{\lambda''(\psi_\alpha)}\dsum_{s=1}^l
\dfrac{p_s^2 P'(\varphi_s)^2}{(q_{l+1}^2-p_s^2)(q_\alpha^2-p_s^2)}  \right),\beta=l+1.
\end{array}\right. \label{fm2.18}
\eeq
\end{lem}

\begin{proof}By the definition of $\lambda(\varphi)$ in \eqref{fm2.9} and using \eqref{fm2.15},
 we get
\beq\frac{c_{\alpha,m} P P'(\varphi)}{P^2-q_\alpha^2}\frac{\lambda'(\varphi)}{\lambda''(\psi_\alpha)}=\p_{u_\alpha}\lambda(\varphi)=
2k\bi \lambda(\varphi)\p_{u_\alpha}\varphi_{l+1}-\dsum_{s=1}^l\dfrac{2p_s
P'(\varphi_s)\lambda(\varphi)}{P^2-p_s^2}\p_{u_\alpha}
\varphi_{s}.
\label{fm2.19} \eeq
Putting $\varphi=\varphi_\beta$ for $\beta=1,\cdots,l$ into \eqref{fm2.19} and using \eqref{fm2.16},
we obtain
$$\p_{u_\alpha}\varphi_\beta
=- \dfrac{c_{\alpha,m} p_\beta  P'(\varphi_\beta)}{\lambda''(\psi_\alpha)\,(p_\beta^2-q_\alpha^2)}, \quad \beta=1,\cdots, l$$
and furthermore,
\eqa
\frac{\p_{u_\alpha}\lambda(\varphi)}{\lambda(\varphi)}&=&
2k\bi \p_{u_\alpha}\varphi_{l+1}-\dsum_{s=1}^l\dfrac{2p_s
P'(\varphi_s)}{P^2-p_s^2}\p_{u_\alpha}
\varphi_{s}\nn\\
&=&2k\bi \p_{u_\alpha}\varphi_{l+1}- \frac{2 c_{\alpha,m} }{\lambda''(\psi_\alpha)}\dsum_{s=1}^l
\dfrac{p_s^2  P'(\varphi_s)^2}{(P^2-p_s^2)(q_\alpha^2-p_s^2)}.
\label{fm2.20}
\eeqa
Putting $\varphi=\psi_\beta$ into \eqref{fm2.20}, then
$$\frac{\delta_{\alpha\beta}}{u_\beta}=2k\bi \p_{u_\alpha}\varphi_{l+1}- \frac{2 c_{\alpha,m} }{\lambda''(\psi_\alpha)}\dsum_{s=1}^l
\dfrac{\,p_s^2  P'(\varphi_s)^2}{(q_\beta^2-p_s^2)(q_\alpha^2-p_s^2)}.$$
Especially, taking $\varphi=\psi_{l+1}$, we obtain the desired formula of $\p_{u_\alpha}\varphi_{l+1}$.
\end{proof}

\begin{lem} For $\beta,\gamma=1,\cdots, l$, we have
\eqa S_{\beta,\gamma}:=\dsum_{\alpha=1}^{l+1}\dfrac{c_{\alpha,m} u_\alpha}{\lambda''(\psi_\alpha)\,(p_\beta^2-q_\alpha^2)(p_\gamma^2-q_\alpha^2)}
=\frac{\delta_{\beta\gamma}}{2\,p^2_\beta\,(p^2_\beta-1)}.\label{fm2.21}\eeqa
\end{lem}
\begin{proof} Letting
\eqa \lambda(z)=(z-1)^{-m}(a_0 z^{k+m}+\cdots+a_l z^{-n})=a_0\, (z-1)^{-m}\,
z^{-n} \prod_{j=1}^l(z-p_j^2). \nn \eeqa
So, $\lambda(\varphi)=\lambda(z)|_{z=P^2}$ and
\eqa
\frac{d\lambda(z)}{dz}=k a_0  (z-1)^{-m-1} z^{-n-1} \prod_{\alpha=1}^{l+1}(z-q_\alpha^2),\quad \dfrac{\lambda(z)}{z (z-1)\frac{d\lambda(z)}{dz}}=\dfrac{\prod_{j=1}^l(z-p_j^2)}{\prod_{\alpha=1}^{l+1}(z-q_\alpha^2)},\nn \eeqa
which yields that if $q_\alpha^2\ne 0$ (or $1$)  for all $\alpha=1,\cdots, l+1$, then $z=0$ (or $1$) is not a pole of the function $\dfrac{\lambda(z)}{{z (z-1) \frac{d\lambda(z)}{dz}}}$.

With the use of \eqref{fm2.17} and $P'(\varphi)^2=1-P^2$,
we rewrite $S_{\beta,\gamma}$ as
\eqa
S_{\beta,\gamma}&=& \left.\dsum_{\alpha=1}^{l+1}\dfrac{\lambda(\varphi) (P^2-q_\alpha^2)}
{P \lambda'(\varphi) P'(\varphi) (P^2-p_\beta^2) (P^2-p_\gamma^2)}\right|_{\varphi=\psi_\alpha}\nn\\
&=& -\frac{1}{2}\dsum_{\alpha=1}^{l+1}\left.\dfrac{\lambda(z) (z-q_\alpha^2)}
{z (z-1) \frac{d\lambda(z)}{dz} (z-p_\beta^2) (z-p_\gamma^2)}\right|_{z=q_\alpha^2}\nn\\
&=&-\frac{1}{2}\dsum_{\alpha=1}^{l+1}\left.\res_{z=q_\alpha^2}\dfrac{\lambda(z)}
{z (z-1)\frac{d\lambda(z)}{dz} (z-p_\beta^2) (z-p_\gamma^2)}\right|_{z=q_\alpha^2}\nn\\
&=& \frac{1}{2}(\res_{z=\infty}+\res_{z=p_\beta^2}+\res_{z=p_\gamma^2})
\dfrac{\lambda(z)}{\frac{d\lambda(z)}{dz} z(z-1)
(z-p_\beta^2)(z-p_\gamma^2)}dz\nn\\
&=& \dfrac{\delta_{\beta\gamma}}{2}\res_{z=p_\beta^2}
\dfrac{\lambda(z)}{\frac{d\lambda(z)}{dz} z(z-1)
(z-p_\beta^2)^2}dz=\dfrac{\delta_{\beta\gamma}}{2 p_\beta^2(p_\beta^2-1)}. \nn
\eeqa
We thus prove the identity \eqref{fm2.21}.\end{proof}

\begin{lem}
\eqa
\frac{\lambda''(\psi_{l+1})}{\lambda(\psi_{l+1})}=-2
\left(k+\dsum_{s=1}^l\frac{p_s^2\,P'(\varphi_s)^2}{(q_{l+1}^2-p_s^2)^2}\right).
\label{fm2.22}
\eeqa
\end{lem}

\begin{proof} Observe that
\beq \lambda'(\varphi)=2PP'(\varphi)\lambda(\varphi)\left(\dsum_{s=1}^l\frac{1}{P^2-p_s^2}
-\frac{m}{P^2-1}-\frac{n}{P^2}\right), \label{zfm2.17}\eeq
which yields
\beq \left.P'(\varphi)\left(\dsum_{s=1}^l\frac{1}{P^2-p_s^2}
-\frac{m}{P^2-1}-\frac{n}{P^2}\right)\right|_{\varphi=\psi_{l+1}}=0.\label{zfm2.18}\eeq

\medskip

\noindent{\bf Case 1. $m=0$}. In this case, $P'(\psi_{l+1})=0$. Using \eqref{zfm2.17} and \eqref{zfm2.18},
we have
\eqa \left.\frac{\lambda''(\varphi)}{\lambda(\varphi)}\right|_{\varphi=\psi_{l+1}}
=2(l-k)-\dsum_{s=1}^l\frac{2q_{l+1}^2}{q_{l+1}^2-p_s^2}=
-2k-\dsum_{s=1}^l\frac{2\,p_s^2}{q_{l+1}^2-p_s^2},\nn
\eeqa
which is exactly the formula \eqref{fm2.22} because of $q_{l+1}=1$ and $P'(\varphi_s)^2=1-q_s^2$.

\medskip

\noindent{\bf Case 2. $m\ne 0$}. In this case, $P'(\psi_{l+1})\ne 0$.
By using \eqref{zfm2.18},
\beq \dsum_{s=1}^l\frac{1}{q_{l+1}^2-p_s^2}=\frac{m}{q_{l+1}^2-1}+\frac{n}{q_{l+1}^2}.\label{zfm2.19} \eeq
So, using \eqref{zfm2.17} and \eqref{zfm2.19},  we get
\eqa \left.\frac{\lambda''(\varphi)}{\lambda(\varphi)}\right|_{\varphi=\psi_{l+1}}&=&
\left.2PP'(\varphi)\dfrac{d}{d\varphi}\left(\dsum_{s=1}^l\frac{1}{P^2-p_j^2}
-\frac{m}{p^2-1}-\frac{n}{p^2}\right)\right|_{\varphi=\psi_{l+1}}\nn\\
&=&-2
\left(k+\dsum_{s=1}^l\frac{p_s^2 P'(\varphi_s)^2}{(q_{l+1}^2-p_s^2)^2}\right). \nn
\eeqa
The lemma is proved.
 \end{proof}

We are now in a position to state our main theorem in this section.

\ian{\begin{thm}\label{Main2}
Let $\mathfrak{h}: \mathcal{M}_{k,m}(C_l) \to \mathfrak{M}_{k,m,n}$ be induced by the map  
 \eqa (x_1,\cdots, x_{l+1})\mapsto (\varphi_1,\cdots,\varphi_{l+1})\label{fm2.10}\eeqa
with
\[\varphi_1=\pi x_1,\quad \varphi_j=\pi(x_j-x_{j-1}),\quad \varphi_{l+1}=\pi x_{l+1}, \quad j=2,\cdots, l.\]
Then $\mathfrak{h}$ is a $k$-fold covering map,  which is also a local isomorphism
between the Frobenius manifolds
$\mathcal{M}_{k,m}(C_l)$ and $\mathfrak{M}_{k,m,n}$.
\end{thm}}

\begin{proof}
\ian{Let us first prove that the  $\mathfrak{h}$ is a $k$-fold covering map. In fact,
by using the formulae \eqref{fm2.8} and \eqref{fm2.9} one has
\eqa   a_0P^{2l}+\dsum_{j=1}^{l} a_j P^{2(l-j)}=a_0\,\prod_{j=1}^l(P^2-p_j^2),
 \quad a_0=e^{2k\bi\,\varphi_{l+1}} \label{add5.23} \eeqa
 and
 \beq
 a_j=(-1)^j a_0\,\sigma_j(p_1^2,\cdots,p_l^2),\quad j=1,\cdots,l.\nn
 \eeq
Observe that $$p_j^2=\cos^2\varphi_j=\dfrac{\cos(2\varphi_j)+1}{2},\quad
\cos(2\varphi_j)=\dfrac{e^{2\bi \varphi_j}+e^{-2\bi \varphi_j}}{2}=\dfrac{\xi_j}{2},$$ and with these one obtains
$$a_0=e^{ky^{l+1}}, \quad a_j=(-\dfrac{1}{4})^j \sigma_j (\xi_1+2,\cdots,\xi_l+2) e^{ky^{l+1}},\quad j=1,\cdots,l.$$
So the map $\mathfrak{h}: (\tilde{y}_1,\dots, \tilde{y}_{l+1})\mapsto
(a_0, a_1,\dots, a_l)$ is given by
\eqa
&&a_0=\tilde{y}_{l+1}^k,\\
&&a_j=(-\dfrac{1}{4})^j\left(\dsum_{s=1}^j 2^{j-s}\binom{l-s}{j-s} \tilde{y}_{l+1}^{k-d_s} \tilde{y}_s+2^j \binom{l}{j} \tilde{y}_{l+1}^k \right)
\label{fm2.12}
\eeqa
for $j=1,\dots,l$, here $d_s=s$ for $s=1,\cdots, k$ and $d_s=k$ for $s=k+1,\cdots,l$,
and the Jacobian of $\mathfrak{h}$ is proportional to ${\tilde{y}_{l+1}}^{k(k+1)/2-1}$.
From the above representation of the map $\mathfrak{h}$ we also have
\eqa
 &&a_l=(-1)^l 4^{-l} \tau^{l},
\quad \textrm{for}\ m=0.\notag\\
&&\sum_{j=0}^l a_j=-(-4)^{-k-1} \tau^l,
\quad \textrm{for}\ m= l-k>0;\notag\\
&& a_l=(-1)^l 4^{m-l} \tau^{l-m},\quad \sum_{j=0}^l a_j=-(-4)^{-k-1} \tau^l,
\quad \textrm{for}\ m\ne 0, m\ne l-k.\notag
\eeqa
Note that the condition $m\ne l-k$ is equivalent to $n\ne 0$.
So from the definition of $\mathcal{M}_{k,m}(C_l)$ and $\mathfrak{M}_{k,m,n}$ given by Theorem \ref{mt1} and \eqref{fm2.1}--\eqref{cnd-a-3} that $\mathfrak{h}$
is a $k$-fold covering map.
}

\ian{Now let us proceed to prove that $\mathfrak{h}$ is a local isomorphism
between the two Frobenius manifold. It is not difficult to check that the Euler vector fields \eqref{fm2.5} and \eqref{zz1} coincide. }
So it suffices to prove that the
intersection form \eqref{fm2.4} coincides with the intersection form
of the orbit space, and the metric \eqref{fm2.3} coincides with the metric \eqref{DS3.18}.

By definition of $\eta$ in \eqref{fm2.3} and using \eqref{fm2.15}, we get
\eqa \eta_{\alpha\beta}(u)&:=&\eta(\p_{u_\alpha},\p_{u_\beta})=(-1)^{k+1}\dsum_{|\lambda|<\infty}
\res_{d\lambda=0}
\dfrac{\p_{u_\alpha}(\lambda(\varphi)d\varphi)\p_{u_\beta}(\lambda(\varphi)d\varphi)}{d\lambda(\varphi)}\nn\\
&=&(-1)^{k+1}\dsum_{\gamma=1}^{l+1}\res_{\varphi=[\psi_\gamma]}
\frac{c_{\alpha,m} c_{\beta,m} P^2 P'(\varphi)^2}{(P^2-q_\alpha^2)(P^2-q_\beta^2)} \frac{\lambda'(\varphi)}{\lambda''(\psi_\alpha)\lambda''(\psi_\beta)} d\varphi.\nn\eeqa
We remark that $[\psi_\gamma]$ represents four different points
$\pm \psi_\gamma$ and $\pm \psi_\gamma+\pi$ satisfying $q_\gamma^2=(e^{\bi [\psi_\gamma]}+e^{-\bi [\psi_\gamma]})^2$.
Obviously, when $\alpha\ne \beta$, $\eta_{\alpha\beta}(u)=0$.
So,
\eqa \eta_{\alpha\alpha}(u)&=&(-1)^{k+1}\res_{\varphi=[\psi_\alpha]}
\frac{c_{\alpha,m}^2 P^2 P'(\varphi)^2}{(q_\alpha^2-P^2)^2} \frac{\lambda'(\varphi)}{\lambda''(\psi_\alpha)^2} d\varphi\nn\\
&=&(-1)^{k}\frac{2c_{\alpha,m}^2}{\lambda''(\psi_\alpha)^2} \res_{P=\pm q_\alpha}
\frac{P^2}{P^2-q_\alpha^2} \frac{\dot{\lambda}(P)(P^2-1)}{P^2-q_\alpha^2} dP
\nn\\
&=&(-1)^{k}\frac{2\,c_{\alpha,m}^2}{\lambda''(\psi_\alpha)^2} \res_{P=\pm q_\alpha}
\frac{P^2}{P^2-q_\alpha^2} \frac{\dot{\lambda}(P)(P^2-1)}{P^2-q_\alpha^2} dP
\nn\\
&=&(-1)^{k+1}\frac{2\,c_{\alpha,m}}{\lambda''(\psi_\alpha)}.\nn
\eeqa
We thus obtain
\eqa \eta_{\alpha\beta}(u)=(-1)^{k+1}
\frac{2\,c_{\alpha,m}\delta_{\alpha\beta}}{\lambda''(\psi_\alpha)}. \nn \eeqa
Similarly, we can obtain  the formula of $g_{\alpha\beta}(u):=g(\p_{u_\alpha},\p_{u_\beta})$ as
\eqa g_{\alpha\beta}(u)=
-\frac{2\,c_{\alpha,m}\delta_{\alpha\beta}}{u_\alpha\lambda''(\psi_\alpha)}. \nn \eeqa
Observe that the vector field $e=\dsum_{j=k}^l c_j\dfrac{\p}{\p y^j}$ in \eqref{unity}
in the coordinates
$a_0,\cdots, a_{l}$ coincides with $e=(-1)^k\dsum_{s=0}^m (-1)^{m-s}\binom{m}{s}
\dfrac{\p}{\p a_{k+m-s}}.$
The shifting
\beq a_{k+m-s}\longmapsto a_{k+m-s}+c\,(-1)^{m-s}\binom{m}{s} ,\quad s=0,\cdots,m,\nn\eeq
produces the corresponding shift
\beq u_\alpha\longmapsto u_\alpha+c, \quad \alpha=1,\cdots,l+1 \nn\eeq
of the critical values. This shift does not change the critical points $\psi_\alpha$
neither the values of the second derivative $\lambda''(\psi_\alpha)$. So
\beq \mathcal{L}_e{g}^{\alpha\beta}= \mathcal{L}_e(-\frac{u_\alpha \lambda''(\psi_\alpha)}{2\,c_{\alpha,m}\delta_{\alpha\beta}})
=(-1)^{k+1}\frac{\lambda''(\psi_\alpha)}{2\,c_{\alpha,m}\delta_{\alpha\beta}}={\eta}^{\alpha\beta}.\label{DS5.25}\eeq

Finally, we compute the metric $g^{\beta\gamma}(\varphi)$ given by
\eqa g^{\beta\gamma}(\varphi):&=&(d\varphi_\beta,d\varphi_\gamma)
=\dsum_{\alpha,\kappa=1}^{l+1}\frac{1}{g_{\alpha\kappa}(u)}\dfrac{\p \varphi_\beta }{\p u_\alpha}\,\dfrac{\p \varphi_\gamma}{\p u_\kappa}=\dsum_{\alpha=1}^{l+1}\frac{1}{g_{\alpha\alpha}(u)}\p_{u_\alpha}\varphi_\beta\,\p_{u_\alpha}\varphi_\gamma.\nn
\eeqa
Using \eqref{fm2.18}, \eqref{fm2.21} and \eqref{fm2.22}, we have\\

\noindent \underline{\bf Case 1.  $1\leq \beta, \gamma\leq l$}.
\eqa
g^{\beta\gamma}(\varphi)=-\dfrac{p_\beta\,p_\gamma \, P'(\varphi_\beta)\,P'(\varphi_\gamma)}{2}\dsum_{\alpha=1}^{l+1}\dfrac{c_{\alpha,m} \,u_\alpha}{\lambda''(\psi_\alpha)\,(p_\beta^2-q_\alpha^2)(p_\gamma^2-q_\alpha^2)}
=\frac{1}{4}\delta_{\beta\gamma}. \nn
\eeqa

\noindent \underline{\bf Case 2.  $1\leq \beta\leq l$ and $\gamma=l+1$}.
\eqa
g^{\beta,l+1}(\varphi)
&=&\dfrac{p_\beta \, P'(\varphi_\beta)}{4k\bi} \dsum_{\alpha=1}^{l+1} \dfrac{1}{p_\beta^2-q_\alpha^2} \left(\delta_{\alpha,l+1}+
\dfrac{2\,c_{\alpha,m} \,u_\alpha}{\lambda''(\psi_\alpha)}\dsum_{s=1}^l
\dfrac{\,p_s^2\,P'(\varphi_s)^2}{(q_{l+1}^2-p_s^2)(q_\alpha^2-p_s^2)}\right)\nn\\
&=& \dfrac{p_\beta \, P'(\varphi_\beta)}{4k\bi} \left(\dfrac{1}{p_\beta^2-q_{l+1}^2}-
\dsum_{s=1}^{l}
\dfrac{\,p_s^2\,P'(\varphi_s)^2}{q_{l+1}^2-p_s^2}
\dsum_{\alpha=1}^{l+1}\dfrac{2\,c_{\alpha,m} \,u_\alpha}{\lambda''(\psi_\alpha)(q_\alpha^2-p_s^2)(q_\alpha^2-p_\beta^2)}\right)\nn\\
&=& 0. \nn
\eeqa

\medskip

\noindent \underline{\bf Case 3. $\beta=\gamma=l+1$}.
\eqa
g^{l+1,l+1}(\varphi)&=&\dsum_{\alpha=1}^{l+1}
\dfrac{u_{\alpha}\,\lambda''(\psi_\alpha)}{8k^2 \, c_{\alpha,m}}\left(\dfrac{\delta_{\alpha,l+1}}{\lambda(\psi_{\alpha})}+\dfrac{2\,c_{\alpha,m} }{\lambda''(\psi_\alpha)}\dsum_{s=1}^l
\dfrac{\,p_s^2\,P'(\varphi_s)^2}{(q_{l+1}^2-p_s^2)(q_\alpha^2-p_s^2)}  \right)^2\nn\\
&=& \frac{1}{8k^2} \frac{\lambda''(\psi_{l+1})}{\lambda(\psi_{l+1})}+\frac{1}{2k^2} \dsum_{s=1}^l\frac{p_s^2\,P'(\varphi_s)^2}{(q_{l+1}^2-p_s^2)^2}\nn\\
&+& \frac{1}{2k^2}\dsum_{s,j=1}^l\frac{p_s^2\,p_j^2\,P'(\varphi_s)^2\,P'(\varphi_j)^2}{(q_{l+1}^2-p_s^2)\,(q_{l+1}^2-p_j^2)} \dsum_{\alpha=1}^{l+1} \dfrac{c_{\alpha,m} \,u_\alpha}{\lambda''(\psi_\alpha)(q_\alpha^2-p_s^2)(q_\alpha^2-p_j^2)}
\nn\\
&=&-\frac{1}{4k^2}\left(k+\dsum_{s=1}^l\frac{p_s^2\,P'(\varphi_s)^2}{(q_{l+1}^2-p_s^2)^2}\right)+\frac{1}{2k^2} \dsum_{s=1}^l\frac{p_s^2\,P'(\varphi_s)^2}{(q_{l+1}^2-p_s^2)^2}\nn\\
&+&\frac{1}{4k^2} \dsum_{s,j=1}^l\frac{p_s^2\,p_j^2\,P'(\varphi_s)^2\,P'(\varphi_j)^2}{(q_{l+1}^2-p_s^2)\,(q_{l+1}^2-p_j^2)}  \frac{\delta_{s j}}{p^2_j\,(p^2_j-q_{l+1}^2)}\nn\\
&=&-\frac{1}{4k}.\nn
\eeqa
Using \eqref{fm2.10}, it is easy to know that the intersection form $g^{\alpha\beta}(\varphi)$
coincides with $(~,~)^{\sptilde}$ defined in \eqref{DS3.1} and \eqref{DS3.2}. The coincidence of the metric
\eqref{fm2.3} with the metric \eqref{DS3.18} follows \eqref{DS5.25}. We thus complete the proof of the theorem.
\end{proof}

\begin{rem}  On the orbit space of the extended affined Weyl group $\widetilde{W}^{(k)}(D_{k+2})$,
Dubrovin and Zhang constructed a \ian{quasi-}homogenous polynomial Frobenius structure,
denoted by  $\mathcal{M}_{\mathrm{DZ}}^{(k)}(D_{k+2})$ which is isomorphic to
$\mathfrak{M}_{k,1,1}$. Actually, in this case, there is a  tri-polynomial description introduced in
\cite{P2010, takahashi}, also used in \cite{DLZ2012}.
\end{rem}

\section{Concluding Remarks}

For the root systems of type $B_l, C_l$ and $D_l$,
we have constructed families of Frobenius manifold structures on the orbit
spaces of the extended affine Weyl groups $\widetilde{W}^{(k)}(R)$ with respect to
the choice of an arbitrary vertex on the Dynkin diagram, as was suggested in
\cite{slodowy}, motivated by the results of \cite{wir,LO1,LO2}.
{In our construction for the root system $C_l$, we perform the following two steps:
\begin{enumerate}
\item[i)] We fix the k-th vertex of the
Dynkin diagram and define an extension of the affine Weyl group, and construct
a symmetric bilinear form $(g^{ij})$ on the cotangent space of the orbit space of the extended affine Weyl group.
\item[ii)] We find a unity vector field $e$ which is labeled by an integer $0\le m\le l-k$,
and construct a Frobenius manifold structure on the orbit space.
\end{enumerate}
We may ask the question whether one can perform the same construction for the root system of type $A_l$? Namely, we can perform the first step as it is done in \cite{DZ1998} for any $1\le k\le l$. For the second step, only one choice of the unity vector field $e$ is given in \cite{DZ1998} to construct a Frobenius manifold structure on the orbit space. Then is there other choices of the unity vector field? The answer is no, i.e. we can not find a different unity vector field satisfying the conditions of Lemma
\ref{lem-du}. It remains a challenging problem to understand whether the constructions of the present
paper can be generalized to the root systems of the types   $E_6$, $E_7$, $E_8$, $F_4$, $G_2$.}

Another open problem is to obtain an explicit realization of the integrable hierarchies associated with the Frobenius
manifolds of the type $\widetilde{W}^{(k)}(R)$. So far this problem was solved only for $R=A_l$, see \cite{CDZ,C2006,
DZ2001, DZ04,MT2008,MST2014} for details. We plan to study these problems in subsequent publications.


Observe that  the potential of the semisimple Frobenius manifold structures
constructed above from the root systems of type $(C_l,k,m=0)$ has  the form
$$
F=\frac12\,(t^k)^2 t^{l+1}+\frac12 t^k \sum_{\al,\beta\ne k} \eta_{\al\beta} t^\al t^\beta+
\sum_{j=0}^n f_j(t^2, t^3,\dots, t^l,\frac1{t^l})\, e^{j\, t^{l+1}},
$$
where $f_j(t^2, t^3,\dots, t^l,\frac1{t^l}), j=0,\dots, n$ are some polynomials of their independent variables.
The Euler vector field has the form
$$
E=\sum_{j=1}^l d_j \frac{\p}{\p t^j}+ r\frac{\p}{\p t^{l+1}}.
$$
Here $0<d_j<1, r> 0$, and they also satisfy the duality relation given in \eqref{zh-n1}, \eqref{zh-n2} for
the case $m=0$. We expect that these potentials of semisimple Frobenius manifolds, together
with the ones that are constructed in \cite{DZ1998}, exhaust all solutions of the
above form, and we have verified this for the cases when $l=1,2,3$ and $n\le 6$.


\vskip 0.3 cm
\noindent {\bf Acknowledgements.} The authors are grateful to the referee for very careful reading of the manuscript and many valuable suggestions.  They thank Si-Qi Liu
for his help on the proof of the lemma \ref{lem3.1}.
The research of Y.Z. is partially supported  by {NSFC (No.11171176,
  No.11371214, No.11471182). } The research of D.Z. is partially supported
by NSFC (No.11671371, No.11871446) and Wu Wen-Tsun Key Laboratory of Mathematics, USTC, CAS.


\end{document}